\documentclass[a4paper,11pt,oneside]{article}
\usepackage[utf8]{inputenc}
\usepackage[T1]{fontenc}
\usepackage[english]{babel}
\usepackage{amsthm}
\usepackage{array}
\usepackage{tikz}
\usepackage{appendix}
\usepackage{amsmath}
\usepackage{amssymb}
\usepackage{hyperref}
\usepackage{mathrsfs} 
\usepackage{bbm}
\newcommand{\xdownarrow}[1]{%
  {\left\downarrow\vbox to #1{}\right.\kern-\nulldelimiterspace}
}
\hypersetup{
    colorlinks=true, 
    linktoc=all,
    urlcolor=darkgray,   
    linkcolor=blue,  
    citecolor=darkgray,
}
\usepackage{babel}
\usepackage{csquotes}
\definecolor{greenpigment}{rgb}{0.0, 0.65, 0.31}
\definecolor{iceberg}{rgb}{0.44, 0.65, 0.82}

\usepackage{graphicx, sidecap}
\usepackage{algorithm}
\usepackage{algpseudocode}
\usepackage{wasysym}
\usepackage{pifont,stmaryrd}

\usepackage{mathtools}
\usepackage{amsfonts}
\usepackage{multicol}
\usepackage{setspace}
\usepackage{listings}
\usepackage{tikz}
\usetikzlibrary{arrows}
\usetikzlibrary{positioning}
\usepackage{caption}
\usetikzlibrary{shapes,arrows}
\usepackage{multirow}
\usepackage{setspace}
\usepackage[top=1.5cm, bottom=3cm, left=2cm , right=1cm]{geometry}
\usepackage{fancybox}
\usepackage{xcolor, color}
\usepackage{pifont}
\reversemarginpar
\usepackage{url}

\usepackage{authblk}

\usepackage[all]{xy}
\usepackage{tikz-cd}

\usepackage{amssymb}

\newcommand{\R}{\mathbb{R}}
\newcommand{\C}{\mathbb{C}}
\newcommand{\N}{\mathbb{N}}
\newcommand{\Z}{\mathbb{Z}}
\newcommand{\E}{\mathbb{E}}
\newcommand{\F}{\mathbb{F}}
\newcommand{\HH}{\mathbb{H}}

\newcommand{\CC}{\mathbf{C}}

\newcommand{\U}{\mathcal{U}}
\newcommand{\DD}{\mathcal{D}}
\newcommand{\g}{\mathfrak{g}}
\newcommand{\kk}{\mathfrak{k}}
\newcommand{\aL}{\mathfrak{a}}
\newcommand{\End}{\text{End}}
\newcommand{\Hom}{\text{Hom}}
\newcommand{\Hol}{\text{Hol}}
\newcommand{\Pol}{\text{Pol}}

\newcommand{\Ker}{\text{Ker}}
\newcommand{\Impart}{\text{Im}}
\newcommand{\Repart}{\text{Re}}

\newcommand{\Id}{\text{Id}}

\newcommand*{\QEDA}{\null\nobreak\hfill\ensuremath{\square}}%
\newcommand{\SL}{\mathbf{SL}}
\newcommand{\SO}{\mathbf{SO}}
\newcommand{\SU}{\mathbf{SU}}

\usepackage{cancel}
\theoremstyle{plain}

\numberwithin{equation}{section}

\theoremstyle{plain}
\newtheorem{cor}{Corollary}

\newtheorem{prop}{Proposition}
\newtheorem{thm}{Theorem}
\newtheorem{hyp}{Hypothesis}
\newtheorem{conj}{Conjecture}
\newtheorem{lem}{Lemma}
\newtheorem{defn}{Definition}
\theoremstyle{definition}

\theoremstyle{remark}
\newtheorem{req}{Remark}

\title{\textbf{Solvability of invariant systems of differential equations on $\HH^2$ and beyond}}
\author[1]{Martin \textsc{Olbrich}}
\author[1]{Guendalina \textsc{Palmirotta}\footnote{corresponding author, e-mail: \href{mailto:guendalina.palmirotta@uni.lu}{guendalina.palmirotta@uni.lu}}}
\affil[1]{University of Luxembourg, Faculty of Science, Technology and Medicine, Department of Mathematics, 6, Avenue de la Fonte,  L-4364 Esch-sur-Alzette,  Luxembourg}
\date{} 

\begin{document}

\maketitle

\abstract
We show how the Fourier transform for distributional sections of vector bundles over symmetric spaces of non-compact type $G/K$ can be used for questions of solvability of systems of invariant differential equations in analogy to Hörmander's proof of the Ehrenpreis-Malgrange theorem.
We get complete solvability for the hyperbolic plane $\HH^2$ and partial results for products $\HH^2 \times \cdots \times \HH^2$ and the hyperbolic 3-space $\HH^3$.\\


\tableofcontents

\section{Introduction}
In the Euclidean case $\R^n$ (for some positive integer $n$), it is well-known by Malgrange \cite{Malgrange} and Ehrenpreis \cite{Ehrenpreis1} that all linear differential operators with constant coefficients
$$
D=\sum_{|\alpha| \leq m} c_\alpha D^\alpha, \;\;\; 
D^\alpha=\frac{\partial^{|\alpha|}}{\partial_{x_1}^{\alpha_1} \cdots \partial_{x_n}^{\alpha_n}},
\; \alpha=(\alpha_1, \dots, \alpha_n), \alpha_j \in \N_0,\; |\alpha|=\sum_{j=1}^n \alpha_j
$$
are solvable in the space of smooth functions $C^\infty(\R^n)$.
Even more, there exists a fundamental solution $E$ for $D$, which is a distribution
on $\R^n$ such that $DE=\delta$ and consequently, the function $f:=E*g$
is a solution of the equation $Df=g$ at least for $g$ with compact support.
Here, $\delta$ denotes the Dirac measure or delta-distribution at the origin on $\R^n$.
Constant coefficients operators on $\R^n$ are exactly invariant differential operators, if we consider $\R^n$ as a Lie group of translations.
Speaking of Lie groups $G$, one may ask the analog question for invariant differential operators, which are left invariant under $G$:
\begin{equation} \label{eq:DOleft}
D(f \circ l_g)=(Df) \circ l_g, \;\;\; f\in C^\infty(G), \forall g \in G,
\end{equation}
where $l_g$ is the left translation on $G$.
Such differential operators are, in general, not even locally solvable.
On the other hand, the situation is much better for non-zero linear differential operators $D$, which are bi-invariant under $G$, this means that \eqref{eq:DOleft} holds not only for the left but also for the right translation on $G$.
In fact, $D$ is locally solvable for simply connected nilpotent Lie groups by Raïs \cite{Rais}, for solvable Lie groups by Raïs-Duflo \cite{DufloRais} and Rouvière \cite{Rouviere}, and for semi-simple Lie groups by Helgason \cite{HelgasonL}.
Some interesting results for more general groups have been found in \cite{Cerezo} and \cite{HelgasonL}.
Now, if we consider the quotient $X=G/K$ with $G$ a non-compact connected semi-simple Lie group with finite center and $K$ is its maximal compact subgroup, then Helgason \cite[Chap.~V]{Helgason3} proved that all 
$G$-invariant differential operators are solvable on Riemannian symmetric spaces $X$ of non-compact type.
The key tool of most of these works is the theory of harmonic analysis on the corresponding Lie group or on symmetric spaces, in particular the application of the Fourier transform.

However, what happens, if we genuinely extend this situation and consider systems of linear invariant differential operators, e.g. $D$ as a $q \times p$ -matrix ($q,p \in \N$) of such linear invariant differential operators? Is it still (locally) solvable?
In case of $\R^n$, the questions have been answered completely by Malgrange \cite{MalgrangeS1,MalgrangeS2}, Ehrenpreis (\cite{Ehrenpreis3} and \cite[Chap. 6]{Ehrenpreis4}) and Palamodov \cite{Palamodov1,Palamodov2}.
Already, here, the proof is much more complicated as for a single operator.

In Section~\ref{sect:StatConj}, we introduce a setting for the above problem for symmetric spaces $X$ of non-compact type in terms of a conjecture, see Conjecture~\ref{conj:1}.
Suitably interpreted, it also applies to the above example, with $X=\R^n$ with $G=\R^n$ and $K=\{0\}$. Indeed, we view an invariant system of invariant differential equations as an invariant differential operator between complex homogeneous vector bundles.
The conjecture involves a second invariant differential operator $\widetilde{D}$ satisfying $\widetilde{D} \circ D=0$ as an integrability condition.
An overview on previously known results covering the conjecture is given at the end of Section~\ref{sect:StatConj}.\\
In Section~\ref{sect:Solv}, we present a possible strategy to solve the conjecture.
The following diagram, which we try to explain in the sequel, pictures the strategy 
and also the `jungle' into which the reader is about to adventure.

\[
\begin{tikzcd}
C^{\infty}(X,\E_{\gamma}) 
\arrow[r, "D"] & 
\underset{\text{exact}}{C^{\infty}(X,\E_{\tau})}
\arrow[d, <->, double, "\text{dualization}"] 
\arrow[r, "\widetilde{D}"] & 
\underset{\text{closed range}}{C^{\infty}(X,\E_{\delta})}
\\
  C^{-\infty}_c(X,\E_{\tilde{\delta}}) 
\arrow[d, "\mathcal{F}_{\tilde{\delta}}"] \arrow[r, "\widetilde{D}^t"] & 
\underset{\text{exact}}{C^{-\infty}_c(X,\E_{\tilde{\tau}})}
\arrow[d, "\mathcal{F}_{\tilde{\tau}}"] 
\arrow[r, "D^t"] & 
\underset{\text{closed range}}{C^{-\infty}_c(X,\E_{\tilde{\gamma}})}
 \arrow[d, "\mathcal{F}_{\tilde{\gamma}}"]  \\
 PWS_{\tilde{\delta}}(\aL^*_\C \times K/M) 
\arrow[d, dashrightarrow, <->, ""]  \arrow[r, "Q"] &
\underset{\text{exact}}{ PWS_{\tilde{\tau}}(\aL^*_\C \times K/M) }
\arrow[d, dashrightarrow, ""] 
\arrow[r, "P"] 
\arrow[d, dashrightarrow, <->, ""] &
\underset{\text{closed range}}{ PWS_{\tilde{\gamma}}(\aL^*_\C \times K/M) }
\arrow[d, dashrightarrow, <->, ""]
&
\arrow[d, dashrightarrow, <->, ""]
(\textcolor{blue}{Level\; 2})\\
\prescript{}{\mu}{PWS}_{\tilde{\delta}}(\aL^*_\C) 
\arrow[r, "\prescript{}{\mu}{Q}"]  &  
\underset{\text{exact}}{\prescript{}{\mu}{PWS}_{\tilde{\tau}}(\aL^*_\C)}
\arrow[r, "\prescript{}{\mu}{P}"] 
 & 
\underset{\text{closed range}}{\prescript{}{\mu}{PWS}_{\tilde{\gamma}}(\aL^*_\C).}
& 
(\textcolor{blue}{Level\;3})
\end{tikzcd}
\]
Roughly speaking, in order to prove the exactness of the first sequence, we dualize.
Here dual bundles corresponding to dual representations are indicated by a tilde.
Then, as next step, we apply the Fourier transform for sections.
The image of the Fourier transform of the vector space $C^{-\infty}_c(X,\E_{\tilde{*}}), *=\delta,\tau,\gamma$, is equal to the Paley-Wiener-Schwartz space $PWS_{\tilde{*}}(\aL^*_\C \times K/M)$.
The precise description of $PWS_{\tilde{*}}(\aL^*_\C \times K/M)$ has been obtained in our paper \cite[Thm.~4]{PalmirottaPWS}.
This theorem has been derived from Delorme's Paley-Wiener theorem for $C^\infty_c(G)$ \cite{Delorme}.
The corresponding operators $Q$ and $P$ are given by multiplication endomorphism by valued polynomials on $\aL^*_\C$.
Nevertheless, the modified problem, which we call (\textcolor{blue}{Level 2}), is still too difficult to be solved in general. 
Therefore, we fix an additional $K$-type on the right. We refer this situation as (\textcolor{blue}{Level 3}).
In this framework, the problem becomes more tractable, in particular, we can use Hörmander's results on multiplication by polynomial matrices.
More precisely, our problem reduces to certain hypotheses (Hyp.~\ref{lem:PQ0L3}-\ref{thm:estimate}), which will turn out to be true at least for some groups $G$.
The final step would be to move back to (\textcolor{blue}{Level 2}) and take the inverse Fourier transform to obtain the solutions of the initial problem. The strategy is summarized in Thm.~\ref{thm:localsolv}, Prop.~\ref{prop:1} and Prop.~\ref{prop:2}.\\
In Section~\ref{sect:H2}, we establish the solvability for $G=\SL(2,\R)$ by using the very explicit description of $\prescript{}{\mu}{PWS}_*(\aL^*_\C)$ given in \cite{PalmirottaIntw}. 
The final result is given in Thm.~\ref{thm:new1}.
Partial results (Thm.~\ref{thm:new2} and Thm.~\ref{thm:Hyp1SL2C}) will be obtained for $G=\SL(2,\R)^d$ $(d \in \N)$ and $G=\SL(2,\C)$ in Section~\ref{sect:Beyond}.\\
We have already checked that the methods introduced in the present paper also lead to a complete resolution of the conjecture for $G=\SO(1,n), n$ arbitrary.
In contrast, as already the example $\SL(2,\R)^d, d>1$, shows, for higher rank symmetric spaces additional new ideas would be needed.\\

This work is part of the doctoral dissertation \cite{PalmirottaDiss} of the second author.

\section{The conjecture} \label{sect:StatConj}
\subsection{Invariant differential operators on sections of homogeneous vector bundles} \label{subsect:DO}
Let $G$ be a real connected semi-simple Lie group with finite center of non-compact type.
Let $K \subset G$ be a maximal compact subgroup. 
The quotient $X=G/K$ is a Riemannian symmetric space of non-compact type.
Let $\tau$ be a finite dimensional representation of $K$ on a complex vector space $E_\tau$.
Let $\E_\tau:=G \times_K E_\tau$ be the $G$-homogeneous vector bundle over $X$ induced by $E_\tau.$
The space of its smooth sections is denoted by $C^\infty(X,\E_\tau)$.
Note that $C^\infty(X,\E_\tau)$ carries a natural Fréchet topology and it is equipped with a smooth $G$-action. 
Moreover, we have the following isomorphism of $G$-modules:
$$C^\infty(X,\E_\tau) \cong C^\infty(G,E_\tau)^K \cong [C^{\infty}(G) \otimes E_\tau]^K.$$
Here, $W^K$ denotes the space of $K$-invariants of a $K$-representation $W.$\\
We also need the $G$-module of distribution sections of $\mathbb{E}_\tau$
$$C^{-\infty}(X,\E_\tau):= C^\infty_c(X,\E_{\tilde{\tau}})' \cong [C^{-\infty}(G) \otimes E_\tau]^K \supset C^\infty(X,E_\tau),$$
where $(\tilde{\tau},E_{\tilde{\tau}})$ is the dual representation of $(\tau, E_\tau).$
We now consider an additional $K$-representation $(\gamma, E_\gamma)$ and its associated homogeneous vector bundle $\E_\gamma$ over $X$.

\begin{defn}
A linear non-zero differential operator
$D: C^\infty(X,\E_\gamma) \rightarrow C^\infty(X,\E_\tau)$
between sections of homogeneous vector bundles is said to be $G$-invariant if
$$D(g \cdot f)= g\cdot (Df), \;\;\;\;\; \forall g \in G, f \in C^\infty(X,\E_\gamma).$$
\end{defn}
We denote by $\DD_{G}(\E_\gamma,\E_\tau)$ the vector space of all these $G$-invariant differential operators.\\
Let $D\in  \DD_{G}(\E_\gamma,\E_\tau)$. As any differential operator, it extends canonically to a differential operator between distribution sections
$$D:C^{-\infty}(X,\E_\gamma) \longrightarrow C^{-\infty}(X,\E_\tau).$$
This extension is again $G$-equivariant.
We look at distribution sections
$$C^{-\infty}_{\{0\}}(X,\E_*) \subset C^{-\infty}(X,\E_*), \qquad *=\gamma, \tau,$$
supported at the origin $o:=eK \in X.$
Let $\g$ (resp. $\kk$) be the Lie algebra of the Lie group of $G$ (resp. $K$) and $\U(\g)$ (resp. $\U(\kk)$) be the universal enveloping algebra of complexification of $\g$ (resp. $\kk$).
$C^{-\infty}_{\{0\}}(X,\E_*)$ carries compatible actions of $\U(\g)$ and $K$.
It consists of derivatives of (vector valued) delta distributions.
It follows that $C^{-\infty}_{\{0\}}(X,\E_*)$ is a finitely generated $\U(\g)$-module consisting only of $K$-finite elements.
Shortly, it is a finitely generated $(\g,K)$-module.
Moreover, it is canonically isomorphic to the $(\g,K)$-module
$\U(\g)\otimes_{\U(\kk)} E_*$, where the $\U(\g)$-action is given by left multiplication on the first factor and $K$ acts by $\mathrm{Ad} \otimes *$. More details are provided in \cite[Sect.~7]{PalmirottaPWS}.
Restriction from $C^{-\infty}(X,\E_\gamma)$ to $C^{-\infty}_{\{0\}}(X,\E_\gamma)$ now induces $(\g,K)$-homomorphisms
$D:C^{-\infty}_{\{0\}}(X,\E_\gamma) \longrightarrow C^{-\infty}_{\{0\}}(X,\E_\tau)$
and hence
\begin{equation} \label{eq:DOisom}
D: \U(\g)\otimes_{\U(\kk)} E_\gamma \longrightarrow \U(\g)\otimes_{\U(\kk)} E_\tau.
\end{equation}
For any $(\g,K)$-module $W$ we have a natural isomorphism
$$\Hom_{(\g,K)}(\U(\g) \otimes_{\U(\mathfrak{k})} E_*, W) \cong \Hom_K(E_*,W), \quad \phi \mapsto h_\phi, \quad h_\phi(v):=\phi(1 \otimes v).$$
In particular, we obtain an isomorphism
\begin{eqnarray} \label{eq:DhD}
    \DD_G(\E_\gamma, \E_\tau) &\cong& \Hom_K(E_\gamma, \U(\g) \otimes_{\U(\kk)} E_\tau) \nonumber \\
    D &\mapsto& h_D.
\end{eqnarray}
The injectivity of the linear map $D\mapsto h_D$ is obvious, for a discussion of surjectivity we refer again to \cite[Sect.~7]{PalmirottaPWS}.
We remark that \eqref{eq:DhD} gives an algebraic description of invariant differential operators which is closely related to the well-known isomorphism \cite{Reimann}
$$\DD_G(\E_\gamma, \E_\tau) \cong [\U(\g) \otimes_{\U(\kk)} \Hom(E_\gamma,E_\tau)]^K.$$
However, it is not trivial to write down directly a well-defined linear map from $[\U(\g) \otimes_{\U(\kk)} \Hom(E_\gamma,E_\tau)]^K$ to $\Hom_K(E_\gamma, \U(\g) \otimes_{\U(\kk)}  E_\tau)$ which gives a correct isomorphism, compare \cite[Prop.~1.5.]{PalmirottaDiss}.

\subsection{Statement of the conjecture} \label{subsect:Conj}
Let $D\in \DD_G(\E_\gamma,\E_\tau)$. 
We wish to find another bundle $\E_\delta$ and another operator $\widetilde{D} \in \DD_G(\E_\tau, \E_\delta)$ such that the equation $Df=g$ is solvable in $C^\infty(X,\E_{\gamma})$ if and only if $\widetilde{D}g=0$.

We look at \eqref{eq:DOisom} and consider the $(\g,K)$-submodule 
$\Ker\; D \subset \U(\g) \otimes_{\U(\kk)} E_\gamma$.
Since $\U(\g)$ is Noetherian \cite[0.6.1.]{Wallach}, $\Ker\; D$ is again finitely generated as $\U(\g)$-module. We choose a finite dimensional $K$-invariant generating subspace $E_0 \subset \Ker\; D$.
It carries a $K$-representation 
$\tau_0: K \rightarrow \mathrm{GL}(E_0).$
Then the embedding
$i: \E_0 \hookrightarrow \U(\g) \otimes_{\U(\kk)} E_\gamma$ belongs to $\Hom_K(E_0, \U(\g) \otimes_{\U(\kk)} E_\gamma)$.
By \eqref{eq:DhD}, the map $i$ determines an invariant differential operator
$D_0 \in \DD_G(\E_0,\E_\gamma)$ with $h_{D_0}=i.$
The image of the corresponding
$(\g,K)$-map
$$D_0: \U(\g) \otimes_{\U(\kk)}E_0 \longrightarrow \U(\g) \otimes_{\U(\kk)} E_\gamma$$
coincides with $\Ker\; D$.
In other words, the sequence
\begin{equation}\label{eq:DDO} 
\U(\g) \otimes_{\U(\kk)} E_0 \stackrel{D_0}{\longrightarrow} \U(\g) \otimes_{\U(\kk)} E_\gamma 
\stackrel{D}{\longrightarrow} \U(\g) \otimes_{\U(\kk)} E_\tau
\end{equation}
is exact, in particular 
$D \circ D_0=0. $
However, since we want the opposite order, we need to dualize.
The dual or `adjoint' of $D$ is an invariant differential operator
$$D^t: C^{\infty}(X,\E_{\tilde{\tau}}) \longrightarrow C^{\infty}(X,\E_{\tilde{\gamma}}).$$
We view $D^t$ as an operator 
$D^t: \U(\g) \otimes_{\U(\kk)} E_{\tilde{\tau}} \longrightarrow \U(\g) \otimes_{\U(\kk)} E_{\tilde{\gamma}}$
and
set $\widetilde{D}:=((D^t)_0)^t$, where $(D^t)_0$ is constructed as above with $\gamma, \tau$ replaced by $\tilde{\gamma}, \tilde{\tau}$.
$(D^t)_0$ corresponds to a finite-dimensional $K$-invariant generating subspace $F_0 \subset \Ker(D^t) \subset \U(\g) \otimes_{\U(\kk)} E_{\tilde{\tau}}$.
Set $\F:=\widetilde{\F_0}$.
Then $\widetilde{D} \in \DD_G(\E_\tau, \F)$ and we have  $(\widetilde{D} \circ D)^t=D^t \circ (D^t)_0=0$.
Thus,
$\widetilde{D} \circ D=0$ on $C^\infty(X, \E_\gamma)$, i.e. $\Impart(D) \subset \Ker(\widetilde{D})$, viewed as a subspace of $C^\infty(X,\E_\tau)$.\\

\begin{conj} \label{conj:1}
Let $D \in \DD_G(\E_\gamma,\E_\tau)$ and $\widetilde{D}\in \DD_G(\E_\tau,\F)$ be as above.\\
Then, the differential equation $Df=g$ is solvable in $C^\infty(X, \E_\gamma)$ for given $g \in C^\infty(X, \E_\tau)$ if and only if $\widetilde{D}g=0$.
In other words, the sequence
\begin{equation}\label{eq:exact}
C^\infty(X, \E_\gamma) \stackrel{D}{\longrightarrow} \underset{\text{exact}}{C^\infty(X, \E_\tau)} \stackrel{\widetilde{D}}{\longrightarrow} C^\infty(X, \F)
\end{equation}
 is exact in the middle, i.e. $\emph \Impart(D)= \emph \Ker(\widetilde{D})$.
\end{conj}

\noindent
Note that one implication is obvious, since $\widetilde{D} \circ D=0$ implies that 
$\widetilde{D}g=\widetilde{D} \circ Df=0.$

The following lemma implies in particular that the validity of Conj.~\ref{conj:1} is independent of the choice of $\widetilde{D}$, i.e. the choice of the generating subspace $F_0.$
Moreover, it says that $\widetilde{D}$ gives the strongest possible integrability condition for the equation $Df=g$ that can be formulated by invariant differential operators.
\begin{lem}
    Let $D \in \DD_G(\E_\gamma, \E_\tau)$ and $\widetilde{D}\in \DD_G(\E_\tau, \F)$ be as in Conj.~\ref{conj:1}.
    Let $D_1 \in \DD_G(\E_\tau, \E_\delta)$ for some $K$-representation $(\delta, E_\delta)$ be such that $D_1 \circ D=0.$
    Then there exists $P\in \DD_G(\F, \E_\delta)$ with $D_1 = P \circ \widetilde{D}.$
\end{lem}

\begin{proof}
    The property $D_1 \circ D=0$ implies $\Impart(D_1^t) \subset \Ker(D^t)$ for the operators
    $D_1^t: \U(\g) \otimes_{\U(\kk)} E_{\tilde{\delta}} \rightarrow \U(\g) \otimes_{\U(\kk)} E_{\tilde{\tau}}$
    and 
    $D^t: \U(\g) \otimes_{\U(\kk)} E_{\tilde{\tau}} \rightarrow \U(\g) \otimes_{\U(\kk)} E_{\tilde{\gamma}}$.
    Recall that $\Ker(D^t)$ coincides with the image of
    $$\widetilde{D}^t: \U(\g) \otimes_{\U(\kk)}F_0 \rightarrow \U(\g) \otimes_{\U(\kk)} E_{\tilde{\tau}}.$$
    We choose a $K$-invariant complement
    $W \subset \U(\g) \otimes_{\U(\kk)}F_0$ of $\Ker(\widetilde{D}^t)$.
    Sending $v\in E_{\tilde{\delta}}$ to the unique $w\in W$ with
    $D_1^t(1 \otimes v)=\widetilde{D}^t(w)$
    defines an element $h \in \Hom_K(E_{\tilde{\delta}}, \U(\g) \otimes_{\U(\kk)}F_0).$
    By \eqref{eq:DhD}, $h$ is of the form $h_Q$ for some $Q \in \DD_G(\E_{\tilde{\delta}}, \widetilde{\mathbb{F}}).$
    By construction $\widetilde{D}^t \circ Q=D_1^t.$
    Set $P:=Q^t,$ then $P \in \DD_G(\mathbb{F},\E_\delta)$ and $D_1=P \circ \widetilde{D}.$
\end{proof}

For many special cases, the conjecture is known to be true (but not in general):
\begin{itemize}
\item Consider $\E_p=\wedge^{p} T^*X$ ($p\in \N_0$) the exterior powers of the cotangent bundle $T^*X$, and \\$D=d_p: C^\infty(X,\E_p) \rightarrow C^\infty(X,\E_{p+1})$ the exterior differential. Then, $\tilde{D}=d_{p+1}$.
By the Poincaré-Lemma (e.g. see \cite[Sect.~4]{Bott}) the conjecture in this situation is true, even for any contractible manifold $X$.
\item Notice that if $\E_\gamma=\E_\tau=\C$ are the trivial one-dimensional vector bundles on $X$ and $D \neq 0$, then the transposed invariant differential operator $D^t$ is injective on $C_c^{- \infty}(X,\C)$.
This implies that $\widetilde{D}=0$, thus we have no `integrability' condition.
Helgason's result \cite[Chap. V]{Helgason3}, mentioned in the introduction, shows that the conjecture is true in this case.
\item Also for elliptic operators between sections of general bundles, we have $\widetilde{D}=0$.
In this case, the conjecture follows from a result by Malgrange \cite[p.~341]{Malgrange}, saying in particular that every elliptic operator with analytic coefficients on a non-compact analytic manifold is surjective.
\item Furthermore, the Euclidean analogue of the conjecture is true by the result of Ehrenpreis (\cite{Ehrenpreis3} and \cite[Chap.~6]{Ehrenpreis4}),
Malgrange \cite{MalgrangeS1,MalgrangeS2} and Palamodov \cite{Palamodov1,Palamodov2} mentioned in the introduction. 
One of the first written proofs appeared in Hörmander \cite[Thm.~7.6.13 and Thm.~7.6.14]{Hormander}.
Note that the first edition of Hörmander's book appeared already in 1966.
For this proof, Hörmander invented new $L^2$-methods.
The result in the Euclidean case is known as a part of the Ehrenpreis Fundamental Principle.
\item An extension of Ehrenpreis's fundamental principle to symmetric spaces is discussed by Oshima, Saburi and Wakayama \cite{Oshima}.
In particular, they annonced that the conjecture is true if
$\E_\gamma=\C^p, \E_\tau=\C^q$ and $\E_\delta=\C^r$, $p,q,r \in \N$, with trivial $K$-representations.
\item Another important result is a consequence of the theory of Kashiwara-Schmid \cite{Kashi,Kashiwara} on maximal globalizations of Harish-Chandra modules.
Let $D \in \DD_G(\E_\gamma,\E_\tau)$.
We look at the operator
$$D^t: \U(\g) \otimes_{\U(\kk)} E_{\tilde{\tau}} \longrightarrow \U(\g) \otimes_{\U(\kk)} E_{\tilde{\gamma}}.$$
If $D^t$ occurs in a projective resolution
\begin{equation} \label{eq:resolution}
    \dots \stackrel{P_{s+1}}{\rightarrow} \U(\g) \otimes_{\U(\kk)} E_{s+1} \stackrel{P_{s}}{\rightarrow} \U(\g) \otimes_{\U(\kk)} E_{s} \stackrel{P_{s-1}}{\rightarrow} \dots \stackrel{P_{0}}{\rightarrow}\; \U(\g) \otimes_{\U(\kk)} E_{0} \rightarrow W \rightarrow 0,
\end{equation}
i.e. $D^t=P_s$ for some $s\in \N_0$, of a Harish-Chandra module $W$ (e.g. an irreducible $(\g,K)$-module W), then the conjecture is true for $D$.
This result, in more abstract language, is announced in \cite{Kashi}.
The full proof appeared in \cite{Kashiwara} and is based on the theory of $\mathcal{D}$-modules.
It is one of the motivations of the present paper to find a `more elementary' proof of this result based on the Fourier transform.
Note also that not all invariant differential operators $D$ satisfy $D^t=P_s$ for some resolution \eqref{eq:resolution} as above.
\end{itemize}

\section{On solvability and general strategy} \label{sect:Solv}

\subsection{Topological Paley-Wiener-Schwartz theorems and intertwining conditions} \label{sect:PWS}
Let us fix a $G$-invariant Riemannian metric on $X$.
Let $(\tau,E_\tau)$ be a finite-dimensional representation of $K$. The corresponding closed ball of radius $r$ centered at the origin $o$ is defined by $\overline{B}_r(o)$.
We consider the space of compactly supported distributional sections 
$C^{- \infty}_c(X, \E_\tau) = \bigcup_{r \geq 0} C^{- \infty}_r(X, \E_\tau)$,
where 
$$C^{-\infty}_r(X, \E_\tau):=\{f \in C^{- \infty}(X, \E_\tau) \;|\; \text{supp}(f) \subset \overline{B}_r(o)\}.$$
In context of the Fourier transform, we refer to this situation as (\textcolor{blue}{Level 2}).
For a second finite-dimensional $K$-representation, we also consider the space of compactly supported $(\gamma, \tau)$-spherical distributions\\
$C^{- \infty}_c(G,\gamma,\tau)= \bigcup_{r \geq 0} C^{- \infty}_r(G,\gamma,\tau)$.
Here, 
\begin{eqnarray*}
C^{- \infty}_r(G,\gamma,\tau):=
\{S \in C^{-\infty}(G,\Hom(E_\gamma,E_\tau) \;|\;  
l_{k_1} r_{k_2} S=\tau(k_2)^{-1}S\gamma(k_1) \;\forall k_1,k_2 \in K,\\
\text{ and } \text{supp}(S) \subset p^{-1}(\overline{B}_r(o))\},
\end{eqnarray*}
where $l_{k_1}, r_{k_2}$ are the left and right regular representations, respectively, and 
$p: G \rightarrow X$ is the natural projection.
Note that $C^{- \infty}_r(G,\gamma,\tau) \cong \Hom_K(E_\gamma, C^{- \infty}_r(X, \E_\tau))$.
In particular, for irreducible $\gamma$, $C^{- \infty}_r(G,\gamma,\tau)$ is just the (infinite-dimensional) multiplicity space for the occurrence of $\gamma$ in the $K$-representation $C^{- \infty}_r(X, \E_\tau)$.
We refer to this situation as (\textcolor{blue}{Level 3}).
The missing (\textcolor{blue}{Level 1}) corresponds to $C^{\pm \infty}_c(G)$, where no $K$-representations are involved. It stands in the background of the theory (see \cite{PalmirottaPWS, PalmirottaIntw}) but will play no role in the present paper.

We fix an Iwasawa decomposition $G=KAN$. 
Let $M:=Z_K(A)$ be the centralizer of $A$ in $K$. 
Now we come to the Fourier transform. 

\begin{defn}[Fourier transform for sections of homogeneous vector bundles] \label{defn:FTsect}
Let $g=\kappa(g)a(g)n(g) \in KAN=G$ be the Iwasawa decomposition of $g \in G$.
We define the exponential function 
$$e^\tau_{\lambda,k}: G \longrightarrow \mathrm{End}(E_\tau), \quad
\lambda \in \aL^*_\C, k\in K,$$ 
by
$e^\tau_{\lambda,k}(g):=\tau(\kappa(g^{-1}k))^{-1} a(g^{-1}k)^{-(\lambda+\rho)} \in \emph \End(E_\tau),
$
where $\rho$ is the half sum of the positive roots of $(\g,\aL)$ counted with multiplicity.
\begin{itemize}
\item(\textcolor{blue}{Level 2}) 
For $T\in C^{-\infty}_c(X,\E_\tau)$, the Fourier transform is given by
$$\mathcal{F}_\tau T(\lambda,k):= \langle T, e^{\tau}_{\lambda, k} \rangle \in E_\tau, \;\;\; (\lambda, k) \in \aL^*_\C \times K/M.$$
\item (\textcolor{blue}{Level 3}) 
For $S \in C^{-\infty}_c(G,\gamma, \tau)$, the Fourier transform is defined by
$$ \prescript{}{\gamma}{\mathcal{F}}_\tau S(\lambda):= \langle S, e^\tau_{\lambda,1} \rangle \in \emph \Hom_M(\E_\gamma, \E_\tau), \lambda \in \aL^*_\C.$$
\end{itemize}
\end{defn}

\noindent
We fix a basis $Y_1, \dots,Y_n$ of $\kk$. Then $Y_\alpha=Y_{\alpha_1} \cdots Y_{\alpha_n},\alpha=(\alpha_1, \cdots, \alpha_n)$ a multi-index, form basis of $\U(\kk)$.
We also introduce
$$C^\infty(\aL^*_\C \times K/M, \E_{\tau|_M}) := \{ f: \aL^*_\C \times K \stackrel{C^\infty}{\longrightarrow} E_\tau \;|\; f(\lambda, km)= \tau(m)^{-1} f(\lambda,k) \}.$$

The Riemannian metric on $X$ induces Euclidean norms $|\cdot|$ on $\aL^*$ and $\aL^*_\C$.
\begin{defn}[Paley-Wiener-Schwartz space for sections in (\textcolor{blue}{Level 2}) and (\textcolor{blue}{Level 3}), \cite{PalmirottaPWS}, Def.~10] 
\label{def:PWSspace}
\noindent
\begin{itemize}
\item[(a)] For $r\geq 0$, let $PWS_{\tau,r}(\aL^*_\C \times K/M)$ be the space of sections
$\psi \in C^\infty(\aL^*_\C \times K/M, \E_{\tau|_M})$ such that:
\begin{itemize}
\item[$(2.i)$] The section $\psi$ is holomorphic in $\lambda \in \aL^*_\C$.
\item[$(2.iis)_r$] (growth condition) For all multi-indices $\alpha$, there exist $N \in \N_0$ and a positive constant $C_{r,N,\alpha}$ such that
$$\|l_{Y_\alpha} \psi(\lambda,k)\|_{E_\tau} \leq C_{r,N,\alpha} (1+|\lambda|^2)^{N+\frac{|\alpha|}{2}} e^{r|\emph \Repart(\lambda)|},\;\;\; k \in K,$$
where $\|\cdot\|_{E_\tau}$ is a $K$-invariant Euclidean norm on the finite-dimensional vector space $E_\tau$. 
\item[$(2.iii)$] The intertwining conditions in \cite[Thm.~2 (D.2)]{PalmirottaPWS} are satisfied.
\end{itemize}
We set $PWS_\tau(\aL^*_\C \times K/M) := \bigcup_{r \geq 0,N \in \N_0} PWS_{\tau,r,N}(\aL^*_\C \times K/M)$
equipped with its natural inductive limit topology.

\item[(b)] By considering an additional $K$-representation $\gamma$, let $\prescript{}{\gamma}{PWS}_{\tau,r}(\aL^*_\C)$ be the space of functions
$$\aL^*_\C \ni \lambda \mapsto \varphi(\lambda) \in \emph\Hom_M(E_\gamma,E_{\tau})$$
such that:
\begin{itemize}
\item[$(3.i)$] The function $\varphi$ is holomorphic in $\lambda \in \aL^*_\C$.
\item[$(3.iis)_r$] (growth condition) There exist $N\in \N_0$ and a positive constant $C_{r,N}$ such that
$$
\|\varphi(\lambda)\|_{\text{op}} \leq C_{r,N} (1+|\lambda|^2)^{N} e^{r|\emph \Repart(\lambda)|},
$$
where $\|\cdot\|_{\text{op}}$ denotes the operator norm on $\emph \Hom_M(E_\gamma,E_{\tau})$.
\item[$(3.iii)$] The intertwining conditions in \cite[Thm.~2 (D.3)]{PalmirottaPWS} are satisfied.
\end{itemize}
\end{itemize}
We set $\prescript{}{\gamma}{PWS}_{\tau}(\aL^*_\C) := \bigcup_{r \geq 0,N \in \N_0} \prescript{}{\gamma}{PWS}_{\tau,r,N}(\aL^*_\C)$
equipped with its natural inductive limit topology.
\end{defn}
We do not write the intertwining conditions explicitly here.
We only mention that each intertwining condition comes from a $G$-invariant subspace in a finite sum of (derived) principal series representations of $G$.
Such a subspace is called an intertwining datum.
The most important intertwining data are just proper submodules in a single reducible principal series representation and intertwining operators between two principal series representations.
Each intertwining datum gives a linear relation between the germs of $\psi$ (\textcolor{blue}{Level 2}) or $\phi$ (\textcolor{blue}{Level 3}) at finitely many points $\lambda_1, \dots, \lambda_s \in \aL^*_\C$ depending only at the values and possibly finitely many derivatives of $\psi, \varphi$ at $\lambda_1, \dots, \lambda_s$.

By considering the Knapp-Stein and Želobenko intertwining operators, we developed a criterion for  real rank one \cite[Thm.~2]{PalmirottaIntw} to check when a subset of these intertwining conditions is already sufficient in order to describe the Paley-Wiener-Schwartz space completely.
In particular, for $\SL(2,\R)^d$ ($d\in \N$) and $\SL(2,\C)$, we obtained explicitly the intertwining conditions in (\textcolor{blue}{Level~2}) and (\textcolor{blue}{Level 3}). These results will be exposed in the upcoming Sections~\ref{sect:H2} and \ref{sect:Beyond}.\\
For now, let us present the topological Paley-Wiener-Schwartz theorem, which has already been proved in \cite{PalmirottaPWS}.
We equip the vector spaces $C^{-\infty}_c(X,\E_\tau)$ as well as $C^{-\infty}_c(G,\gamma, \tau)$ with the strong dual topology.

\begin{thm}[Topological Paley-Wiener-Schwartz theorem, \cite{PalmirottaPWS}, Thm.~3] \label{thm:PWSsect}
\noindent
\begin{itemize}
\item[(a)] For each $r\geq 0$, the Fourier transform $\mathcal{F}_\tau$ is a linear
bijection between the two spaces $C^{-\infty}_r(X,\E_\tau)$ and the Paley-Wiener-Schwartz space $PWS_{\tau,r}(\aL^*_\C \times K/M)$.
Moreover, it is a linear topological isomorphism from $C^{-\infty}_c(X,\E_\tau)$ onto $PWS_{\tau}(\aL^*_\C \times K/M)$.

\item[(b)] Similarly, if we consider an additional finite-dimensional $K$-representation $(\gamma,E_\gamma)$ with associated homogeneous vector bundle $\E_\gamma$.
Then, the Fourier transform $\prescript{}{\gamma}{\mathcal{F}}_{\tau}$ is a linear
bijection between the two spaces
$C^{-\infty}_r(G,\gamma, \tau)$ and $\prescript{}{\gamma}{PWS}_{\tau,r}(\aL^*_\C)$, for each $r\geq 0$, and a linear topological isomorphism from $C_c^{-\infty}(G,\gamma, \tau)$ onto $\prescript{}{\gamma}{PWS}_{\tau}(\aL^*_\C)$. \QEDA
\end{itemize}
\end{thm}

Note that $C^{-\infty}_{r=0} (G, \gamma, \tau) \cong \DD_G(\E_\gamma, \E_\tau)$.
Thus, we can apply the Fourier transform to invariant differential operators. Then, Thm.~\ref{thm:PWSsect} implies \cite[Prop.~9]{PalmirottaPWS}:
\begin{eqnarray} \label{eq:PolyPWS}
\prescript{}{\gamma}{\mathcal{F}}_\tau(\DD_G(\E_\gamma,\E_\tau)) 
&\cong& \prescript{}{\gamma}{PWS}_{\tau,0}(\aL^*_\C) \nonumber \\
&=&\{P \in \Pol(\aL^*_\C, \Hom_M(E_\gamma, E_\tau)) \;|\; P \text{ satisfies } (3.iii) \text{ of Def.~\ref{def:PWSspace}}\}.
\end{eqnarray}
In addition, we have the following result.
\begin{prop}[\cite{PalmirottaPWS}, Prop.~10] \label{prop:1}
Let $D \in \DD_G(\E_\gamma, \E_\tau)$ be an invariant linear differential operator.
For $f \in C^{- \infty}_c(X,\E_{\gamma})$ we have that
\begin{equation*} \label{eq:Thm1}
\mathcal{F}_\tau (Df)(\lambda,k)= \prescript{}{\gamma}{\mathcal{F}}_\tau D(\lambda) \mathcal{F}_{{\gamma}} f(\lambda,k), \;\;\;\;\;  \lambda \in \aL^*_\C, k\in K,
\end{equation*}
where $\prescript{}{\gamma}{\mathcal{F}}_\tau D \in \emph \Pol(\aL^*_\C, \emph \Hom_M(E_\gamma, E_\tau))$ is a polynomial in $\lambda \in \aL^*_\C$ with values in $ \emph \Hom_M(E_\gamma, E_\tau)$. \QEDA
\end{prop}

\subsection{Estimates for systems of polynomials equations} \label{sect:Estimate}

The initial goal is to show the exactness in the middle of the sequence \eqref{eq:exact}. 
By taking the topological dual of \eqref{eq:exact}, we are dealing with compactly supported distributional sections $C^{-\infty}_c(X,\E_{\tilde{*}})$ for $*\in \{\delta, \tau, \gamma\}$ in the strong topology.
The strategy is to apply the Fourier transform and the topological Paley-Wiener-Schwartz Thm.~\ref{thm:PWSsect} for sections of homogeneous vector bundles to prove the exactness of the dual sequence in the middle as well as the closed range of our transposed invariant differential operators $D^t \in \DD_G(\E_{\tilde{\tau}},\E_{\tilde{\gamma}})$.
Our strategy could be visualized as follows:
\begin{equation} \label{diag:Level2}
\begin{tikzcd}
C^{\infty}(X,\E_{\gamma}) 
\arrow[r, "D"] & 
\underset{\text{exact}}{C^{\infty}(X,\E_{\tau})}
\arrow[d,<->, double, " "] 
\arrow[r, "\widetilde{D}"] & 
\underset{\text{closed range}}{C^{\infty}(X,\E_{\delta})}
\\
C^{-\infty}_c(X,\E_{\tilde{\delta}}) 
\arrow[d, "\mathcal{F}_{\tilde{\delta}}"] \arrow[r, "\widetilde{D}^t"] & 
\underset{\text{exact}}{C^{-\infty}_c(X,\E_{\tilde{\tau}})}
\arrow[d, "\mathcal{F}_{\tilde{\tau}}"] 
\arrow[r, "D^t"] & 
\underset{\text{closed range}}{C^{-\infty}_c(X,\E_{\tilde{\gamma}})}
 \arrow[d, "\mathcal{F}_{\tilde{\gamma}}"]  \\
 PWS_{\tilde{\delta}}(\aL^*_\C \times K/M) 
\arrow[r, "Q"] & 
\underset{\text{exact}}{PWS_{\tilde{\tau}}(\aL^*_\C \times K/M)}
\arrow[r, "P"] & 
\underset{\text{closed range}}{PWS_{\tilde{\gamma}}(\aL^*_\C \times K/M)}.
\end{tikzcd}
\end{equation}
Moreover, in \eqref{eq:PolyPWS}, we have seen that the Fourier transform of an invariant differential operator, which we denote by $Q$ respectively $P$, is a polynomial in $\lambda\in \aL^*_\C$ with values in the corresponding homomorphism space.
Therefore, we can reformulate the initial Conj.~\ref{conj:1} in terms of action of $Q$ on $PWS_{\tilde{\tau}}(\aL^*_\C \times K/M)$ and action of $P$ on $PWS_{\tilde{\gamma}}(\aL^*_\C \times K/M)$, respectively. \\
However, the problem remains still difficult.
The idea is to get rid of the $K$-variable by fixing an additional irreducible $K$-representation $(\mu,E_\mu)$ on the left while on the right a $K$-representation is given by the bundle $\E_{\tilde{*}} \rightarrow X$, $*\in \{\delta,\gamma,\tau\}$.
In terms of our framework of the previous subsection, we moved from (\textcolor{blue}{Level~2}) to (\textcolor{blue}{Level 3}).
In (\textcolor{blue}{Level 3}), Hörmander's results and estimates
\cite[Thm. 7.6.11. and Cor. 7.6.12.]{Hormander} become applicable.
Thus, the plan is first to solve the main problem in (\textcolor{blue}{Level 3}), as illustrated by the following diagram:
\[
\begin{tikzcd}
PWS_{\tilde{\delta}}(\aL^*_\C \times K/M) 
\arrow[d, dashrightarrow,<->, ""]  \arrow[r, "Q"] &
\underset{\text{exact}}{ PWS_{\tilde{\tau}}(\aL^*_\C \times K/M) }
\arrow[r, "P"] 
\arrow[d, dashrightarrow, <->, ""] &
\underset{\text{closed range}}{ PWS_{\tilde{\gamma}}(\aL^*_\C \times K/M) }
\arrow[d, dashrightarrow, <->, ""]
&
\arrow[d, dashrightarrow, <->, ""]
(\textcolor{blue}{Level\; 2})\\
\prescript{}{\mu}{PWS}_{\tilde{\delta}}(\aL^*_\C) 
\arrow[r, "Q"]  &  
\underset{\text{exact}}{\prescript{}{\mu}{PWS}_{\tilde{\tau}}(\aL^*_\C)}
\arrow[r, "P"] 
 & 
\underset{\text{closed range}}{\prescript{}{\mu}{PWS}_{\tilde{\gamma}}(\aL^*_\C).}
& 
(\textcolor{blue}{Level\;3})
\end{tikzcd}
\]

Let $\prescript{}{\mu}{PWS}_{*,H}(\aL^*_\C) := \{ \varphi \in \Hol(\aL^*_\C, \Hom_M(\E_\mu,\E_*)) \;|\; \varphi$ satisfies
 $(3.iii)$ in Def.~\ref{def:PWSspace}$\}$.
Note that $\prescript{}{\mu}{PWS}_*(\aL^*_\C) \subset \prescript{}{\mu}{PWS}_{*,H}(\aL^*_\C)$.
The analog of Hörmander's result \cite[Lem. 7.6.5]{Hormander} for the Euclidean case, is formulated as the following hypothesis.

\begin{hyp} \label{lem:PQ0L3}
With the notations above, let $P \in \prescript{}{\tilde{\tau}}{PWS}_{\tilde{\gamma},0}(\aL^*_\C)$ and $Q \in \prescript{}{\tilde{\delta}}{PWS}_{\tilde{\tau},0}(\aL^*_\C)$ 
 be the Fourier transforms of the invariant differential operators $D^t \in  \DD(\E_{\tilde{\tau}},\E_{\tilde{\gamma}})$ and $\widetilde{D}^t \in  \DD(\E_{\tilde{\delta}},\E_{\tilde{\tau}}),$ respectively.
Let $f \in \prescript{}{\mu}{PWS}_{\tilde{\tau},H}(\aL_\C^*).$\\
Then, there exists
$g \in \prescript{}{\mu}{PWS}_{\tilde{\delta},H}(\aL_\C^*)$
such that 
$f=Qg$
if and only if
$Pf=0$.
In other words, we have $$\emph \Impart(Q)= \emph \Ker(P) \text{ in } \prescript{}{\mu}{PWS}_{\tilde{\tau},H}(\aL_\C^*).$$
\end{hyp}

Now we want to solve the equation $Pv=w$ in $\prescript{}{\mu}{PWS}_{\tilde{\tau}}(\aL^*_\C)$ with estimates.
For $\varphi \in \prescript{}{\mu}{PWS}_{\tilde{\tau},H}(\aL^*_\C)$, we set
$\|\varphi\|_{r,N}:= \sup_{\lambda \in \aL_\C^*} \Big\{(1+|\lambda|^2)^{-N} e^{-r| \Repart(\lambda)|} \|\varphi\|_{\text{op}} \Big\} \in [0, \infty]$.

\begin{hyp}\label{thm:estimateL3}
Let $P \in \prescript{}{\tilde{\tau}}{PWS}_{\tilde{\gamma},0}(\aL^*_\C)$.
There exist constants $M \in \N_0$ and $C_{r,N}$ 
so that for each function
$\prescript{}{\mu}{u} \in \prescript{}{\mu}{PWS}_{\tilde{\tau},H}(\aL^*_\C)$ with
$\|P\prescript{}{\mu}{u}\|_{r,N} < \infty$, one can find a function $ \prescript{}{\mu}{v} \in \prescript{}{\mu}{PWS}_{\tilde{\tau}}(\aL^*_\C)$ satisfying
\begin{itemize}
\item[(i)] $P\prescript{}{\mu}{u}=P\prescript{}{\mu}{v}$ and
\item[(ii)]
$\|\prescript{}{\mu}{v}\|_{r,N+M} \leq C_{r,N} \|P\prescript{}{\mu}{u}\|_{r,N}.$
\end{itemize}
\end{hyp}

Ideally, we would like to have the constants more or less independent on the $K$-type $\mu$.

\begin{hyp} \label{hyp:constants}
With the notations of Hyp.~\ref{thm:estimateL3}, the constant $M$ can be chosen uniform for all $K$-types $\mu$ and $C_{r,N}$ is of at most polynomial growth in $|\mu| \in [0,\infty)$, the length of the highest weight of $\mu$, with polynomial degree independent of $N$.
\end{hyp}

Now we want to formulate the hypothesis in (\textcolor{blue}{Level 2}) analogous to Hyp.~\ref{thm:estimateL3} in (\textcolor{blue}{Level 3}).
Let $PWS_{\tilde{\tau},H}(\aL^*_\C \times K/M):=\{ \psi \in C^\infty(\aL^*_\C \times K/M,\E_{\tilde{\tau}|_M}) \;|\; \psi$ satisfies $(2.iii)$ in Def.~\ref{def:PWSspace}$\}$.
For $\psi \in PWS_{\tilde{\tau},H}(\aL^*_\C \times K/M)$ we define
$$\| \psi\|_{r,n, \beta} :=\sup_{\lambda \in \aL^*_\C, k \in K} \Big\{(1+|\lambda|^2)^{-\big(N+\frac{|\beta|}{2}\big)} e^{-r|\Repart(\lambda)|} |l_{Y_\beta}\varphi(\lambda,k)| \Big\}  \in [0, \infty], \;\;\; r \geq 0, N \in \N_0, \beta \text{ multi-index}.$$
Here and in the following, we denote the norm $\|\cdot\|_{E_{\tilde{\tau}}}$ simply by $|\cdot|$.

\begin{hyp}\label{thm:estimate}
Let $P \in \prescript{}{\tilde{\tau}}{PWS}_{\tilde{\gamma},0}(\aL^*_\C)$.
There exist constants $M \in \N_0$, 
$k \in \N_0$ and $C_{r,N, \alpha}$
so that for each function
$u\in PWS_{\tilde{\tau},H}(\aL^*_\C \times K/M)$ with $\|Pu\|_{r,N,\beta} < \infty$ for all  multi-indices $\beta$, 
one can find a function $v\in PWS_{\tilde{\tau}}(\aL^*_\C \times K/M)$ satisfying
\begin{itemize}
\item[(i)] $Pu=Pv$ and
\item[(ii)]
$\|v\|_{r,N+M,\alpha} \leq C_{r,N, \alpha} \sum_{|\beta| \leq |\alpha|+k} \|Pu\|_{r,N,\beta}$ for all multi-indices $\alpha.$
\end{itemize}
\end{hyp}

\noindent
Now we can move back from (\textcolor{blue}{Level 3}) to  (\textcolor{blue}{Level 2}).

\begin{thm} \label{thm:Hyp4.4}
Assume that Hyp.~\ref{thm:estimateL3} and Hyp.~\ref{hyp:constants} are satisfied.
Then, Hyp.~\ref{thm:estimate} holds true.
\end{thm}

Before being able to prove Thm.~\ref{thm:Hyp4.4} we need some preparation.
Consider, the space of $L^2$-sections $V:=L^2(K/M,\E_{\tau|_M})$.
Let $\widehat{K}$ denote the set of equivalence classes of irreducible representations of $K$.
Then, by Peter-Weyl's theorem (e.g. see \cite[Thm. 2.8.2.]{Wallach1}), $f= \sum_{\mu \in \widehat{K}} d_\mu \overline{\chi}_\mu * f$, where $d_\mu:=\text{dim}(E_\mu)$ and $\chi_\mu(k):= \text{Tr}(\mu(k))$ is the character of $\mu$. The convolution is with respect to the Haar measure on $K$ normalized by $\int_K dk=1$.
The convergence is in the $L^2$-sense.
If $f \in C^\infty(K/M,\E_{\tau|_M})$, then the convergence is pointwise and uniform (compare Lem.~\ref{lem:absconv} below).
Note that $\overline{\chi}_\mu (k)= \chi_\mu(k^{-1})$.
We fix an orthonormal basis $\{e_i, i=1, \dots,d_\mu\}$ of $E_\mu$. Let $\{\tilde{e}_i, i=1, \dots,d_\mu\}$ be the dual basis of $E_{\tilde{\mu}}$.

We obtain for $f\in C^\infty(K/M,\E_{\tau|_M})$ and $x\in K$:
\begin{eqnarray*}
f(x)= \sum_{\mu \in \widehat{K}} d_\mu \sum_{i=1}^{d_\mu} \int_K \langle \mu(k^{-1}) e_i,\tilde{e}_i \rangle f(k^{-1}x) \;dk 
&=&\sum_{\mu \in \widehat{K}} d_\mu \sum_{i=1}^{d_\mu}  \int_K \langle \mu(kx^{-1}) e_i,\tilde{e}_i \rangle f(k) \;dk \\
&=&\sum_{\mu \in \widehat{K}} d_\mu \sum_{i=1}^{d_\mu}  \int_K f(k) \tilde{\mu}(k^{-1}) \tilde{e}_i \; dk \;\mu(x^{-1})e_i.
\end{eqnarray*}
Hence, we view
$\prescript{}{\mu}{f}_i:= \int_K f(k) \tilde{\mu}(k^{-1}) \tilde{e}_i \; dk$ as an element of $\Hom_M(E_\mu,E_\tau).$
Now consider $w \in PWS_{\tau}(\aL^*_\C \times K/M)$. Then, its Fourier decomposition is given by
\begin{equation} \label{eq:Fourierdecomp}
w(\lambda,k)= \sum_{\mu \in \widehat{K}} d_\mu \sum_{i=1}^{d_\mu} \prescript{}{\mu}{w}_i(\lambda)\mu( k^{-1}) e_i, \;\;\; (\lambda, k) \in \aL^*_\C \times K,
\end{equation}
where 
\begin{equation} \label{eq:Fouriercoeffw}
\lambda \mapsto \prescript{}{\mu}{w}_i(\lambda)= \int_K w(\lambda,k) \tilde{\mu}(k^{-1}) \tilde{e}_i \; dk
 \in  \prescript{}{\mu}{PWS}_{\tau}(\aL^*_\C),  \;\; \forall i=1,\dots,d_\mu
\end{equation}
 are the Fourier coefficients.

\begin{lem}  \label{lem:absconv}
Fix $r\geq 0$ and $N \in \N_0$.
Let $\prescript{}{\mu}{w} _i \in \prescript{}{\mu}{PWS}_{\tau}(\aL^*_\C), \mu \in \widehat{K}, i=1,\dots,d_\mu$,
 be the Fourier coefficients of 
$w\in PWS_{\tau}(\aL^*_\C \times K/M)$.
Consider the Casimir operator of $K$:
$$\CC_\kk=-\sum^{d_\mu}_{i=1} Y_i^2,$$
where $\{Y_i\}$ is an orthonormal basis of $\kk$.
Then, for each integer $p$ we have
\begin{equation} \label{eq:EstimateStep1}
\|\prescript{}{\mu}{w}_i\|_{r,N+p} \leq (1+|\mu|^2)^{-p} \|w\|_{r,N,Z_p}, \;\; \forall i,
\end{equation}
where $Z_p:=(1+\CC_\kk)^p \in \U(\kk)$.
\end{lem}
\noindent
Here, $\|\cdot\|_{r,N,Z_p}$ is defined analogously to $\|\cdot\|_{r,N,\beta}$.

\begin{proof}
For $w \in PWS_{\tau,H}(\aL^*_\C \times K/M)$ and
$\mu \in \widehat{K}$, we can estimate
\begin{eqnarray} \label{eq:lwEstimate2pi}
\|\prescript{}{\mu}{w}_i(\lambda)\|_{\text{op}} 
\stackrel{\eqref{eq:Fouriercoeffw}}{\leq} \int_K \|w(\lambda,k) \tilde{\mu}(k^{-1}) \tilde{e}_i\|_{\text{op}}  \; dk
=\int_K |w(\lambda,k)|  \; dk
\leq  \sup_{k\in K} |w(\lambda,k)| =: \|w(\lambda,\cdot)\|_{K,\infty}.
\end{eqnarray}
Note that the operator norm of $\tilde{\mu}(k^{-1})$ is one.
The Casimir operator acts on $E_\mu$ by the constant \cite[Prop. 5.28 (b)]{Knapp1}
$$\prescript{}{\mu}{\CC_\kk}= |\mu + \rho_k|^2-|\rho_k|^2,$$
where we have identified $\mu$ with his highest weight and $\rho_k$ stands for the half sum of the positive roots of $\kk$.
Since $|\mu +\rho_k|^2-|\rho_k|^2 = |\mu|^2+2\langle \mu, \rho_k \rangle \geq |\mu|^2$, we then have that
$(1+|\mu|^2) \leq (1+\prescript{}{\mu}{\CC}_\kk).$
Hence by \eqref{eq:lwEstimate2pi} applied to $(1+\CC_\kk)^p w$ we obtain for each $p\in \N_0$:
$$
(1+|\mu|^2)^p \| \prescript{}{\mu}{w_i}(\lambda)\|_{\text{op}} 
\leq (1+ \prescript{}{\mu}{\CC}_\kk)^p \| \prescript{}{\mu}{w_i}(\lambda)\|_{\text{op}}
= \| \prescript{}{\mu}{\big((1+\CC_\kk)^p w\big)_i(\lambda)}\|_{\text{op}} 
\leq
\|(1+\CC_\kk)^p w(\lambda,k)\|_{K,\infty}.$$
Now by multiplying the last inequality with the weight factor
$(1+|\lambda|^2)^{-(N+p)}e^{-r|\Repart(\lambda)|}$
 on both sides and taking the supremum over $\lambda \in \aL^*_\C$, we finally get \eqref{eq:EstimateStep1}.
\end{proof}

\begin{proof}[Proof of Thm.~\ref{thm:Hyp4.4}]
Consider $Pu=w \in  PWS_{\tilde{\gamma}}(\aL^*_\C \times K/M)$ 
satisfying $\|w\|_{r,N,\beta} < \infty$ for some $r \geq 0, N \in \N_0$ and all multi-indices $\beta$,
with Fourier decomposition  \eqref{eq:Fourierdecomp}.
Then $\prescript{}{\mu}{w}_i =P \prescript{}{\mu}{u_i}$.
By Hyp.~\ref{thm:estimateL3} and Hyp.~\ref{hyp:constants}, there exists $\prescript{}{\mu}{v}_i \in \prescript{}{\mu}{PWS}_{\tilde{\tau}}(\aL^*_\C)$ such that
 $P\prescript{}{\mu}{v}_i=\prescript{}{\mu}{w}_i$ and with estimate
\begin{equation} \label{eq:estimatew}
\|\prescript{}{\mu}{v}_i\|_{r,N+M} \leq C_{r,N}(1+|\mu|^2)^d \|\prescript{}{\mu}{w}_i\|_{r,N}
\end{equation}
for some constants $M,d \in \N_0$ independent of $N, \mu$ and for some positive constant $C_{r,N}$ independent of $\mu.$
Now choose $p\in \N_0$ so that
$\sum_{\mu \in \widehat{K}} d_\mu^2 (1+|\mu|^2)^{d-p} < \infty.$
Note that $d_\mu$ is of polynomial growth in $\mu$.
We consider the Fourier series
$$v(\lambda,k):= \sum_{\mu \in \widehat{K}} d_\mu \sum_{i=1}^{d_\mu} \prescript{}{\mu}{v}_i(\lambda) \mu(k^{-1})e_i$$
and we want to show that it converges absolutely in the Fréchet space $PWS_{\tilde{\tau},r,N+M+p}(\aL^*_\C \times K/M)$.
Moreover, we will establish the estimates required by Hyp.~\ref{thm:estimate}.
Note that it suffices to consider the semi-norms $\|\cdot\|_{r,N+M+p,Z_l}$.
For $l \in \N_0$. we have
\begin{eqnarray*}
\|v\|_{r,N+M+p,Z_l} 
&\leq& \sum_{\mu \in \widehat{K}} d_\mu \sum_{i=1}^{d_\mu} \|\prescript{}{\mu}{v}_i\|_{r,N+M+p+l} (1+|\mu|^2)^l \\
&\stackrel{\eqref{eq:estimatew}}{\leq} & C_{r,N+p+l}
\sum_{\mu \in \widehat{K}} d_\mu(1+|\mu|^2)^{d+l}  \sum_{i=1}^{d_\mu} \|\prescript{}{\mu}{w}_i\|_{r,N+p+l}\\
&\stackrel{\text{Lem.}~\ref{lem:absconv}}{\leq} & C_{r,N+p+l} \Big(\sum_{\mu \in \widehat{K}} d^2_\mu C (1+|\mu|^2)^{-(p+l)} (1+|\mu|^2)^{d+l} \Big) \|w\|_{r,N,Z_{p+l}} \\
&\leq& C' \|w\|_{r,N,Z_{p+l}},
\end{eqnarray*}
where $C'$ is some positive constant.
\end{proof}

\subsection{On closed range and density} \label{sect:ClosureDensity}

We want to show that $\Impart(D)=\Ker(\widetilde{D})$, which means that $D$ has a closed and dense range in $\Ker(\widetilde{D}) \subset C^\infty(X, \E_\tau)$.
These properties will follow from analogous properties of the transposed operators $D^t$ and $\widetilde{D}^t$, respectively.
We treat the density and closedness separately. 
We start with density.
Let us recall some notions and important results from functional analysis.
We consider a continuous linear operator
$A:V_1 \longrightarrow V_2$ between locally convex Hausdorff vector spaces and its adjoint $A^t:V'_2 \rightarrow V'_1.$
For a linear subspace $W \subset V$, where $V$ is a locally convex Hausdorff vector space, we consider its annihilator $W^\perp \subset V'$ given by
$$W^\perp:=\{\tilde{v}\in V' \;|\; \langle \tilde{v},w\rangle =0, \;\forall w \in W \}.$$
Similarly, for $\widetilde{W} \subset V'$ a linear subspace,
$\widetilde{W}^\perp:=\{v \in V \;|\; \langle \tilde{w}, v\rangle =0, \;\forall \tilde{w} \in \widetilde{W} \}.$
Let $V'_\sigma$ be $V'$ equipped with the weak-* topology.
It is well-known that $(V'_\sigma)'=V$ (as vector spaces) and hence by the Hahn-Banach theorem
$$(\widetilde{W}^\perp)^\perp = \overline{\widetilde{W}} \;\; \text{ (weak-* closure)}.$$
We can now easily derive the following well-known elementary results.

\begin{lem} \label{lem:FA}
With the notations above, we have
\begin{itemize}
\item[(1)] $\emph \Impart(A)^\perp = \emph \Ker(A^t)$.
\item[(2)] $\emph\Ker(A)^\perp = \overline{\emph \Impart(A^t)}$ (weak-* closure). \QEDA
\end{itemize}
\end{lem}

\begin{lem} \label{lem:densityseq}
We consider two operators between locally convex Hausdorff vector spaces
$$V_1 \stackrel{A}{\longrightarrow} V_2 \stackrel{B}{\longrightarrow} V_3$$
with $B \circ A=0$, i.e. $\emph \Impart(A) \subset \emph \Ker(B)$.\\
Then, $\emph \Impart(A) \subset \emph \Ker(B)$ is dense if and only if $\emph \Impart(B^t) \subset \emph \Ker(A^t)$ is weak-* dense.
\end{lem}

\begin{proof}
By the Hahn-Banach theorem, we have that $\Impart(A) \subset \Ker(B)$ is dense if and only if $\Impart(A)^\perp = \Ker(B)^\perp$.
Hence by applying Lem.~\ref{lem:FA}, this is equivalent to $\Ker(A^t)= \overline{\Impart(B^t)}$, i.e.
$\Impart(B^t) \subset \Ker(A^t)$ is weak-* dense.
\end{proof}

\begin{cor} \label{cor:densityseq}
Assume that $\emph \Impart(B^t) \subset \emph\Ker(A^t)$ is dense in the strong dual topology.
Then, $\emph\Impart(A) \subset \emph\Ker(B)$ is dense.
\end{cor}
\begin{proof}
Since strong density implies weak-* density, we directly obtain the result by applying Lem.~\ref{lem:densityseq}.
\end{proof}

\begin{req}
If $V_2$ is (semi-)reflexive, then we even get an equivalence in Cor.~\ref{cor:densityseq} by the Hahn-Banach theorem.
\end{req}

\begin{prop} \label{prop:densityrange}
Assume that Hyp.~\ref{lem:PQ0L3} and Hyp.~\ref{thm:estimateL3} are true.
Then $D \in \DD_G(\E_\gamma,\E_\tau)$ has dense range in $\emph \Ker(\widetilde{D}) \subset C^\infty(X, \E_\tau)$.
\end{prop}

\begin{proof}
Let $P$ and $Q$ as in Hyp.~\ref{lem:PQ0L3}.
In view of Cor.~\ref{cor:densityseq} and the Paley-Wiener-Schwartz Thm.~\ref{thm:PWSsect}, it suffices to show that
$\Impart(Q) \subset \Ker(P) \subset PWS_{\tilde{\tau}}(\aL^*_\C \times K/M)$
is dense.
Consider $f\in \Ker(P)$ with its Fourier decomposition:
$$f(\lambda,k)=\sum_{\mu \in \widehat{K}} d_\mu \sum_{i=1}^{d_\mu} \prescript{}{\mu}{f}_i(\lambda)\mu(k^{-1})e_i, \;\;\; (\lambda,k) \in \aL^*_\C \times K, e_i \in E_\mu,$$
where $\prescript{}{\mu}{f}_i \in \prescript{}{\mu}{PWS}_{\tilde{\tau}}(\aL^*_\C).$
It can be approximated by finite Fourier series
$$f_n(\lambda,k)=\sum_{|\mu| \leq n} d_\mu \sum_{i=1}^{d_\mu} \prescript{}{\mu}{f}_i(\lambda)\mu(k^{-1})e_i, \;\;\; (\lambda,k) \in \aL^*_\C \times K/M, e_i \in E_\mu, \forall n.$$
By applying Hyp.~\ref{lem:PQ0L3}, we have $ \prescript{}{\mu}{f}_i =Q \prescript{}{\mu}{\tilde{g}}_i$, for some $\prescript{}{\mu}{\tilde{g}}_i \in \prescript{}{\mu}{PWS}_{\tilde{\delta},H}(\aL^*_\C)$.
By Hyp.~\ref{thm:estimateL3} we find $\prescript{}{\mu}{g}_i$ in the Paley-Wiener-Schwartz space $\prescript{}{\mu}{PWS}_{\tilde{\delta}}(\aL^*_\C)$ such that $Q \prescript{}{\mu}{g}_i=f_i$.
Thus, $\prescript{}{\mu}{f}_i \in \Impart(Q)$, and hence $f_n \in \Impart(Q)$.
This implies that $\Impart(Q) \subset \Ker(P)$ is dense.
\end{proof}

Now we discuss closedness of $\Impart(D)$.
By the well-known criterion for closed range (e.g. \cite[Chap.~1, Thm.~2.16 (ii)]{Helgason3}  or \cite[ Thm.~7.7]{Schaeffer}) we have to show that
\begin{equation} \label{eq:closedim}
\Impart(D^t) \subset C^{-\infty}_c(X,\E_{\tilde{\gamma}})
\end{equation}
has closed range in the weak-* topology.
For this it suffices to show that for every $B' \subset C^{-\infty}_c(X, \E_{\tilde{\gamma}})$, which is weak-* bounded and weak-* closed that \cite[Chap. 1, Thm.~2.17 (i) and Prop. 2.8]{Helgason3}
\begin{equation} \label{eq:new}
B' \cap \Impart(D^t) \;\; \text{ is weak-* closed}.
\end{equation}
In other words, $B' \cap \Impart(D^t)$ is closed in $B'_\sigma$, where $B'_\sigma$ is $B'$ equipped with the weak-* topology.\\

First we need a compactness lemma.
Let $V$ be the corresponding Paley-Wiener-Schwartz space $PWS_{\tilde{*}}(\aL^*_\C \times K/M)$, $*\in \{\delta,\gamma, \tau\},$ and $V_{r,N}:=PWS_{*,r,N}(\aL^*_\C \times K/M)$. Then
$$V= \bigcup_{r\geq 0}\bigcup_{N \in \N_0} V_{r,N},$$
where
$V_{r,N}$ is a Fréchet space with the generating system of semi-norms $\|\cdot\|_{r,N,\alpha},\alpha \in \N^n_0$, and $V$ carries the corresponding locally convex inductive limit topology.

\begin{lem} \label{lem:compactness}
Let $B \subset V_{r,N}$ be a bounded subset with respect to the corresponding Fréchet topology.
Then $B$ is relatively compact in $V_{r,N+1}$, i.e. the closure of $B$ in $V_{r,N+1}$ is compact.
\end{lem}

\begin{proof}
We consider the space
$$W:=C^\infty(\aL^*_\C \times K/M, \E_{\tau|_M}) \subset C^\infty(\aL^*_\C \times K, E_{\tau})$$
with its natural Fréchet topology.
As the space of smooth sections of a vector bundle, it is also a Montel space.
We have a continuous injection $V_{r,N} \hookrightarrow W$.
Thus $B$, viewed as a subspace of $W$, is relatively compact.\\
Now let $(f_n)$ be a sequence in $B$. We have to show that it has a subsequence $(f_{n_l})$ that converges in $V_{r,N+1}.$

By the above remarks, $(f_n)$ has a subsequence $(f_{n_l})$ that converges in $W$ to some $f\in W$, i.e. $(f_{n_l})$ and all its derivatives converge uniformly on sets of the form $Q \times K \subset \aL^*_\C \times K$, where $Q \subset \aL^*_\C$ is compact.
This implies that $f$ is holomorphic in $\lambda$ and satisfies the $(r,N)$-growth condition as well as Delorme's intertwining conditions, i.e. $f \in V_{r,N} \subset V_{r,N+1}$.\\
However, $(f_{n_l})$ may not converge to $f$ in the $(r,N)$-topology.
Nevertheless, we will argue that it converges in the $(r,N+1)$-topology.
Let $g\in V_{r,N}$ and $\alpha$ be a multi-index.
Then for any $R >0$, we have
\begin{equation} \label{eq:maxg}
\|g\|_{r,N+1,\alpha} \leq \max\Big(\|l_{Y_\alpha} g\|_{\overline{B}_R(0) \times K, \infty} \; , \frac{1}{1+R^2} \|g\|_{r,N,\alpha}\Big),
\end{equation}
where $\overline{B}_R(0)$ denotes the closed ball of radius $R$ in $\aL^*_\C$.
Let $$C:= \sup_{g \in B} \|g\|_{r,N,\alpha} < \infty.$$
Then, also $\|f\|_{r,N,\alpha} \leq C$ and thus
$\|f_{n_l}-f\|_{r,N,\alpha} \leq 2C.$\\
Let $\epsilon >0$. Choose $R$ large enough such that $\frac{2C}{1+R^2} \leq \epsilon$ and $l_0$ large enough such that
\begin{eqnarray*}
\|l_{Y_\alpha}(f_{n_l}-f)\|_{r,\overline{B}_R(0) \times K, \infty} \leq \epsilon \;\;\; \text{ for } l\geq l_0.
\end{eqnarray*}
This is possible since $f_n \rightarrow f$ in $W$.
Now \eqref{eq:maxg}, shows that $\|f_{n_l}-f\|_{r,N+1,\alpha} \stackrel{l \rightarrow \infty}{\longrightarrow} 0$. 
Hence, $f_{n_l} \rightarrow f$ in $V_{r,N+1}$.
\end{proof}

\begin{thm} \label{thm:closure}
Assume that Hyp.~\ref{thm:estimate} is true.
Then, $D \in~\DD_G(\E_\gamma, \E_\tau)$ has closed range in $C^\infty(X, \E_\tau)$.
\end{thm}

\begin{proof}
We consider $P=\prescript{}{\tilde{\tau}}{\mathcal{F}}_{\tilde{\gamma}}(D^t)$ as an operator
$P: PWS_{\tilde{\tau}}(\aL^*_\C \times K/M) \longrightarrow PWS_{\tilde{\gamma}}(\aL^*_\C \times K/M).$
Using Hyp.~\ref{thm:estimate}, we first show that for every $N \in \N_0$ and $r \geq 0$
$$ \Impart(P) \cap PWS_{\tilde{\gamma},r,N}$$
is closed in the $(r,N)$-topology.\\
In fact, consider a sequence $w_n:=Pu_n$ in $\Impart(P) \cap PWS_{\tilde{\gamma},r,N}$ so that $w_n$ converges to $w\in PWS_{\tilde{\gamma},r,N}$, whenever $n$ tends to $\infty$.
Then $(w_n)$ is a Cauchy sequence in $PWS_{\tilde{\gamma},r,N}$.\\
By Hyp.~\ref{thm:estimate}, there exists a sequence $v_n \in PWS_{\tilde{\tau},r,N+M}$ so that
$Pv_n=w_n$ and with estimate $(ii)$ in Hyp.~\ref{thm:estimate}.
This implies that $(v_n)$ is a Cauchy-sequence in $PWS_{\tilde{\tau},r,N+M}$. 
Since,
$PWS_{\tilde{\tau},r,N+M}$ is complete, there exists $v\in PWS_{\tilde{\tau},r,N+M}$ such that
$v_n \stackrel{n \rightarrow \infty}{\longrightarrow} v$ in $PWS_{\tilde{\tau},r,N+M}$.\\
Furthermore, since $P: PWS_{\tilde{\tau},r,N+M} \longrightarrow PWS_{\tilde{\gamma},r,N+M+\alpha}$ is continuous, for $2d \geq \text{deg}(P)$, $w_n=Pv_n$ converges to $Pv$ in $PWS_{\tilde{\gamma},r,N+M+\alpha}$.
But $w_n \longrightarrow w$ also in $PWS_{\tilde{\gamma},r,N+M+ \alpha}$.
Hence, we have that $Pv=w$, i.e., $w\in \Impart(P) \cap PWS_{\tilde{\gamma},r,N}$. 
We conclude that $\Impart(P) \cap PWS_{\tilde{\gamma},r,N}$ is closed in the $(r,N)$-topology.

In order to show the theorem, we need \eqref{eq:new}.
Let $B' \subset C^{-\infty}_c(X, \E_{\tilde{\gamma}})$ be bounded and weak-* closed.
$B'$ is also strongly bounded (this holds in the dual of any Fréchet space, \cite[Thm. 2.17 (ii)]{Helgason2} and of course also strongly closed.
Since $C^{-\infty}_c(X,\E_{\tilde{\gamma}}),$ as the dual of a Montel space, is itself Montel in the strong dual topology, we conclude that $B'_\beta$, i.e.  $B'$ equipped with the strong dual topology, is compact.\\
Now, assume, for a moment, that $B'_\beta \cap \Impart(D^t)$ is closed, then it is compact.
Since the identity map
$$B'_\beta \longrightarrow B'_\sigma$$
is continuous, we conclude that $B'_\sigma \cap \Impart(D^t)$ is compact, i.e. $B' \cap \Impart(D^t)$ is compact in the weak-* topology, in particular it is weak-* closed.\\
Thus, it suffices to show that $B' \cap \Impart(D^t)$ is closed in the strong dual topology, for every strongly bounded and closed $B'$.

We will now apply the Fourier transform and obtain
$$\mathcal{F}(B' \cap \Impart(D^t)) = \mathcal{F}(B') \cap \Impart(P).$$
By the continuity of $\mathcal{F}$ we have that
$$\mathcal{F}(B') \subset  PWS_{\tilde{\gamma}}(\aL^*_\C \times K/M)$$
is compact, in particular closed.
By the definition of the locally convex inductive limit topology, we see that $\mathcal{F}(B') \cap PWS_{\tilde{\gamma},r,N}$ is closed in every $ PWS_{\tilde{\gamma},r,N}$ with respect to $(r,N)$-topology.\\
On the other hand, 
there exist $r \geq 0$ and $N_0\in \N_0$ such that $\mathcal{F}(B') \subset PWS_{\tilde{\gamma},r,N_0}$ and
$\mathcal{F}(B')$ is bounded in the $(r,N_0)$-topology (see the proof of Lem.~3 in \cite{PalmirottaPWS}).
By Lem.~\ref{lem:compactness}
$$\mathcal{F}(B') \subset PWS_{\tilde{\gamma},r,N_0+1}$$
is relatively compact.
But
$\mathcal{F}(B') = \mathcal{F}(B') \cap PWS_{\tilde{\gamma},r,N_0+1}$
is closed in the $(r,N_0+~1)$-topology.
Hence, it is compact in the $(r,N_0+1)$-topology.\\
Now let $N:=N_0+1$.
By the beginning of the proof, $\Impart(P) \cap PWS_{\tilde{\gamma},r,N}$ is closed in the $(r,N)$-topology.
It follows that
$$\mathcal{F}(B') \cap \Impart(P)  = \Impart(P)  \cap PWS_{\tilde{\gamma},r,N} \cap \mathcal{F}(B')$$ 
is compact in the $(r,N)$-topology.
Since the injection 
$$PWS_{\tilde{\gamma},r,N} \hookrightarrow PWS_{\tilde{\gamma}}(\aL^*_\C \times K/M) $$
is continuous, we conclude that $\mathcal{F}(B') \cap \Impart(P)$ is compact, in particular closed, in $PWS_{\tilde{\gamma}}(\aL^*_\C \times K/M) $.
By applying the inverse Fourier transform, which is a topological isomorphism by Thm.~\ref{thm:PWSsect}, we find that $B' \cap \Impart(D^t)$ is strongly closed in $C^{-\infty}_c(X,\E_{\tilde{\gamma}})$. The theorem follows.
\end{proof}

\noindent
Eventually, we can state our main criterion for solvability on $X$, i.e. for the validity of Conj.~\ref{conj:1}.

\begin{thm} \label{thm:localsolv}
Let $D: C^\infty(X, \E_\gamma) \longrightarrow C^\infty(X, \E_\tau)$.
Assume that $D$ statisfies Hyp.~\ref{lem:PQ0L3} and Hyp.~\ref{thm:estimate}.
Then Conj.~\ref{conj:1} is true for $D$.
This holds in particular, if $D$ satisfies Hyp.~\ref{lem:PQ0L3}, Hyp.~\ref{thm:estimateL3} and Hyp.~\ref{hyp:constants}.
\end{thm}

\begin{proof}
Note that Hyp.~\ref{thm:estimate} implies Hyp.~\ref{thm:estimateL3}. 
Vice versa, Thm.~\ref{thm:Hyp4.4} says that 
Hyp.~\ref{thm:estimateL3} and  Hyp.~\ref{hyp:constants} implies Hyp.~\ref{thm:estimate}.
The result now follows from Prop.~\ref{prop:densityrange} and Thm.~\ref{thm:closure}.
\end{proof}


\subsection{Variants of Conjecture~\ref{conj:1}} \label{sect:variants}
Let $D \in \DD_G(\E_\gamma, \E_\tau)$ with corresponding `integrability condition' given by some $\widetilde{D} \in \DD_G(\E_\tau, \E_\delta)$ as in Subsect.~\ref{subsect:Conj}.
One is not only interested in solving equations of the form $Df=g$ on $X$ but also on subsets of $X$ or under additional conditions on $f$ and $g$.
In this subsection, we want to discuss some of these problems which can also be treated by our Hyps. \ref{lem:PQ0L3}-\ref{thm:estimate}.
For this, let us introduce the following notions.
\begin{defn} \label{defn:variants}
\begin{itemize}
\item[(a)] We say that $D$ is solvable on $X$ if Conj.~\ref{conj:1} is true for $D$, i.e. \eqref{eq:exact} is exact.
\item[(b)] We say that $D$ is solvable on $K$-finite sections, if the sequence of $K$-finite elements
$$C^\infty(X, \E_\gamma)_K \stackrel{D}{\longrightarrow} C^\infty(X, \E_\tau)_K \stackrel{\tilde{D}}{\longrightarrow}
C^\infty(X, \E_\delta)_K$$
is exact.
\item[(c)] We say that $D$ is solvable on balls if for all $R>0$ and $x_0 \in X$, the sequence
$$C^\infty(B_R(x_o), \E_\gamma) \stackrel{D}{\longrightarrow} C^\infty(B_R(x_o), \E_\tau) \stackrel{\tilde{D}}{\longrightarrow}
C^\infty(B_R(x_o), \E_\delta)$$
is exact.
\item[(d)] Let $\mathcal{E}_*, *=\gamma, \tau, \delta$, denote the sheaf of smooth sections of $\E_*$ on $X$.
We say that $D$ is locally solvable if the sequence of sheaves
$$ \mathcal{E}_\gamma \stackrel{D}{\longrightarrow} \mathcal{E}_\tau \stackrel{\tilde{D}}{\longrightarrow} \mathcal{E}_\delta$$
is exact, i.e. it is exact at the level of germs.
\end{itemize}
\end{defn}

It is clear that solvability on $X$ implies solvability on $K$-finite sections.
Moreover, solvability on balls implies local solvability.
In the latter two cases, by $G$-invariance, it suffices to check exactness on balls centered at $x_0=o$ and on germs at $x_0=o$, respectively.
Note that in the literature sometimes appear slightly different notions of local solvability.

Our first observation is that the same hypotheses that imply solvability on $X$ also imply solvability on balls and therefore local solvability.

\begin{prop} \label{prop:1}
Assume that $D \in \DD_G(\E_\gamma, \E_\tau)$ satisfies Hyp.~\ref{lem:PQ0L3} and Hyp.~\ref{thm:estimate} (which is the case if it satisfies Hyp.~\ref{lem:PQ0L3}, \ref{thm:estimateL3} and \ref{hyp:constants}).
Then, $D$ is solvable on balls, and hence also locally solvable.
\end{prop}

\begin{proof}
We have to show that 
$$C^\infty(B_R(o), \E_\gamma) \stackrel{D}{\longrightarrow} C^\infty(B_R(o), \E_\tau) \stackrel{\tilde{D}}{\longrightarrow}
C^\infty(B_R(o), \E_\delta)$$
is exact.
Note that the dual spaces
$C^{-\infty}_c(B_R(o), \E_*), *=\gamma, \tau, \delta$, are given by $\bigcup_{0 \leq r < R} C^{-\infty}_r(X, \E_*)$.
By the first part of the Paley-Wiener-Schwartz Thm.~\ref{thm:PWSsect} (a), the Fourier transform is a linear isomorphism between $C^{-\infty}_r (X, \E_*)$ and $PWS_{*,r}(\aL^*_\C \times K/M).$
Taking now the limit over all $0 \leq r < R$ (instead of $0 \leq r < \infty$) one shows as in \cite[Prop.~8]{PalmirottaPWS} that one gets a topological isomorphism
$$C^{-\infty}_c(B_R(o), \E_*) \cong \bigcup_{0 \leq r < R} PWS_{*,r}(\aL^*_\C \times K/M),$$
whose left hand side carries the strong dual topology, and the right hand side is equipped with the locally convex inductive limit topology.
With this change one can argue as in the proof of Prop.~\ref{prop:densityrange} and Thm.~\ref{thm:closure} to establish Prop.~\ref{prop:1}.
\end{proof}

Our second point is that for $K$-finite solvability, it suffices to establish Hyp.~\ref{lem:PQ0L3} and Hyp.~\ref{thm:estimateL3}.

\begin{prop} \label{prop:2}
Assume that $D \in \DD_G(\E_\gamma, \E_\tau)$ satisfies Hyp.~\ref{lem:PQ0L3} and Hyp.~\ref{thm:estimateL3}.
Then, $D$ is solvable on $K$-finite sections.
\end{prop}

\begin{proof}
We have for $*=\gamma, \tau, \delta$
$$C^\infty(X, \E_*)_K = \bigoplus_{\mu \in \widehat{K}} C^\infty(G, \mu, *) \otimes E_\mu.$$
Note that we have the algebraic direct sum here, and that $D$ acts as $D \otimes \Id$ on the right hand side.
Therefore, it is sufficient to show that for every $\mu \in \widehat{K}$
$$C^\infty(G,\mu,\gamma) \stackrel{D}{\longrightarrow} C^\infty(G,\mu,\tau) \stackrel{\widetilde{D}}{\longrightarrow} C^\infty(G,\mu,\delta) $$
is exact.
For this we argue as in the proofs of Prop.~\ref{prop:densityrange} and Thm.~\ref{thm:closure}, where we can work completely at (\textcolor{blue}{Level 3}).
We only have to discuss the sequence
$$\prescript{}{\mu}{PWS}_{\tilde{\delta}}(\aL^*_\C) \stackrel{Q}{\longrightarrow} \prescript{}{\mu}{PWS}_{\tilde{\tau}}(\aL^*_\C) \stackrel{P}{\longrightarrow} \prescript{}{\mu}{PWS}_{\tilde{\gamma}}(\aL^*_\C).$$
We do not have to switch to (\textcolor{blue}{Level 2}).
Therefore, Hyp.~\ref{hyp:constants} or \ref{thm:estimate} is not needed.
\end{proof}

\section{On the solvability on the hyperbolic plane $\HH^2$} \label{sect:H2}
Let $G=\SL(2,\R)$ and $K=\SO(2)$, which is a maximal compact subgroup of $G$.
Then $$G/K \cong \HH^2 = \{z \in \C \;|\; \Impart(z) >0\},$$ where $G$ acts on $\HH^2$ by
$g \cdot z= \begin{pmatrix} a & b \\ c& d \end{pmatrix} \cdot z = \frac{az+b}{cz+d},$ for $g\in \SL(2,\R)$ and $z \in \HH^2.$
We have $\widehat{K} \cong \Z$, where for $m \in \Z$ the corresponding irreducible representation $(m, E_m)$ on $\E_m \cong \C$ is given by
$$\begin{pmatrix} 
\cos\theta & \sin\theta \\
-\sin\theta  &\cos\theta
\end{pmatrix} \mapsto e^{in\theta} \in \mathbf{U}(1).$$
In \cite[Sect.~4]{PalmirottaIntw}, we have described explicitly the intertwining conditions in (\textcolor{blue}{Level 3}) for $\SL(2,\R)$.
Note that $\Hom_M(E_n,E_m)$ and hence $\prescript{}{n}{PWS}_m(\aL^*_\C)$ vanish for $n \not\equiv m$ (mod 2).
If $n \equiv m$ (mod 2), we identify $\Hom_M(E_n,E_m)$ with $\C$.

\begin{thm}[Intertwining conditions in (\textcolor{blue}{Level 3}), \cite{PalmirottaIntw}, Def.~4 and Thm.~6] \label{thm:Meta3SL2R}
Let $n,m \in \Z$ such that $n \equiv m$ (mod 2).
Then, $\varphi \in \emph \Hol(\aL^*_\C,\emph \Hom_M(E_n,E_m))$ satisfies the intertwining  condition (3.iii) in Def.~\ref{def:PWSspace} if and only if there exists an even holomorphic function $h\in \emph \Hol(\lambda^2)$ such that 
\begin{equation} \label{eq:intcondLevel3}
\varphi(\lambda)= h(\lambda) \cdot q_{n,m}(\lambda), \;\;\; \lambda \in \aL^*_\C,
\end{equation}
where $q_{n,m}$ is the polynomial in $\lambda \in \aL^*_\C$ with values in $\emph \Hom_M(E_n,E_m)$ given by
\begin{equation} \label{eq:qnm}
q_{n,m}(\lambda):= \begin{cases}
1, & \text{ if } n=m,\\
(\lambda + \frac{|m|+1}{2})(\lambda + \frac{|m|+3}{2}) \cdots (\lambda + \frac{|n|-1}{2}), & \text{ if } |n|>|m| \text{ and same signs},\\
(\lambda - \frac{|n|+1}{2})(\lambda - \frac{|n|+3}{2})  \cdots (\lambda - \frac{|m|-1}{2}), & \text{ if } |n|<|m| \text{ and same signs},\\
(\lambda + \frac{|n|-1}{2})(\lambda + \frac{|n|-3}{2})  \cdots (\lambda - \frac{|m|-1}{2}), & \text{ else, with different signs.} \QEDA
\end{cases}
\end{equation} 
\end{thm}

Thm.~\ref{thm:Meta3SL2R} tells us that $\prescript{}{n}{PWS}_{m,H}(\aL^*_\C)$ is freely-generated by $q_{n,m}$ as a $\Hol(\lambda^2)$-module.
We will need some information on the product of two such generators.
In fact, $q_{n,m} \cdot q_{l,n} \in \prescript{}{l}{PWS}_{m,0}(\aL^*_\C)$ and thus by Thm.~\ref{thm:Meta3SL2R}:
\begin{equation} \label{eq:rlnm}
q_{n,m} \cdot q_{l,n}= r^{l}_{n,m} \cdot q_{l,m}
\end{equation}
for some even polynomial $r^l_{n,m} \in \Hol(\lambda^2)$.
The following can be checked in a straightforward way.

\begin{lem} \label{lem:rlnm}
Let $n,m,l \in \Z$ be integers so that $n \equiv m$ (mod 2).
Then, for $l,n,m$, we have
\begin{equation*}
r^l_{n,m}=
\begin{cases}
1,& \text{if } l \leq n \leq m \text{ or } m \leq n \leq l,\\
q_{n,m}\cdot q_{m,n},& \text{if } l \leq m < n   \text{ or } n < m \leq l,\\
q_{n,l}\cdot q_{l,n}, & \text{if } m<l<n \text{ or } n<l<m.
\end{cases}
\end{equation*}
In particular, for fixed $n,m$, the set of polynomials $\{r^l_{n,m} \;|\; l \in \Z, l \equiv m$ (mod 2)$\}$ is finite. \QEDA
\end{lem}

\noindent
Consider three, not necessary irreducible, $K$-representations $(\delta,E_{\delta}), (\tau,E_{\tau}) $ and $(\gamma,E_\gamma)$ with 
\begin{equation*} \label{eq:VB}
E_\delta= \bigoplus_{k=1}^{d_\delta} E_{s_k}, \;\;\; 
E_\tau= \bigoplus_{j=1}^{d_\tau} E_{n_j}, \;\;\; 
E_\gamma= \bigoplus_{i=1}^{d_\gamma} E_{m_i},
\end{equation*}
where $s_k,n_j$ and $m_i$ are integers, $\forall k,j,i$.
Now fix an additional irreducible $K$-type $(l,E_l)$.
Let $D \in \DD_G(\E_{\tilde{\gamma}}, \E_{\tilde{\tau}})$ and $P,Q$ as in Hyp.~\ref{lem:PQ0L3}.
Note that we have interchanged the rules of $\gamma, \tilde{\gamma}$, etc.
By using Thm.~\ref{thm:Meta3SL2R} one can now show Hyp.~\ref{lem:PQ0L3}.

\begin{thm} \label{thm:Hyp1SL2R}
The sequence
\begin{equation} \label{eq:seqPWSH}
 \prescript{}{l}{PWS}_{\delta,H}(\aL^*_\C) \stackrel{Q}{\longrightarrow} \prescript{}{l}{PWS}_{\tau,H}(\aL^*_\C) \stackrel{P}{\longrightarrow} \prescript{}{l}{PWS}_{\gamma,H}(\aL^*_\C)
\end{equation}
is exact in the middle, that means that  $\emph \Impart(Q)=\emph \Ker(P)$.
\end{thm}

\begin{proof} 
We look at the defining exact sequence \eqref{eq:DDO} for $\widetilde{D}^t$.
Taking the $l$-isotopic component and applying the Fourier transform, we obtain the exactness of
\begin{equation} \label{eq:new1}
 \prescript{}{l}{PWS}_{\delta,0}(\aL^*_\C) \stackrel{Q}{\longrightarrow} \underset{\text{exact}}{\prescript{}{l}{PWS}_{\tau,0}(\aL^*_\C)} \stackrel{P}{\longrightarrow} \prescript{}{l}{PWS}_{\gamma,0}(\aL^*_\C)
\end{equation}
by Thm.~\ref{thm:PWSsect}.
Note that $\prescript{}{l}{PWS}_{*,0}(\aL^*_\C)$ is a subspace of $\prescript{}{l}{PWS}_{*,H}(\aL^*_\C)$ for $*\in \{\delta,\tau,\gamma\}$.
Moreover due Thm.~\ref{thm:Meta3SL2R}, we have that each $\prescript{}{l}{PWS}_{*,H}(\aL^*_\C)$, which is a direct sum of spaces of the form $\prescript{}{l}{PWS}_{m,H}(\aL^*_\C)$, is a free finitely generated $\Hol(\lambda^2)$-module with polynomial generators.
Hence, $\prescript{}{l}{PWS}_{*,0}(\aL^*_\C)$ is a free $\Pol(\lambda^2)$-module with the same generators.
Thus
\begin{equation} \label{eq:new2}
\prescript{}{l}{PWS}_{*,H}(\aL^*_\C) \cong \Hol(\lambda^2) \otimes_{\Pol(\lambda^2)} \prescript{}{l}{PWS}_{*,0}(\aL^*_\C).
\end{equation}
In addition, $\Hol(\lambda^2)$ is a torsion free module over the principal ideal domain $\Pol(\lambda^2)$. 
This implies that $\Hol(\lambda^2)$ is a flat $\Pol(\lambda^2)$-module.
Hence, tensoring \eqref{eq:new1} with $\Hol(\lambda^2)$ over the ring $\Pol(\lambda^2)$ preserves the exactness of the sequence. 
By \eqref{eq:new2}, we obtain the exactness of \eqref{eq:seqPWSH}.
\end{proof}

Next, we establish Hyp.~\ref{thm:estimateL3} and Hyp.~\ref{hyp:constants}.

\begin{thm}[Estimate result in (\textcolor{blue}{Level 3})] \label{thm:Hyp23SL2R}
Let $P\in \prescript{}{\tau}{PWS}_{\gamma,0}(\aL^*_\C)$ be as above and $l \in \Z$.\\
There exist $M\in \N_0$ and, for all $r \geq 0$ and $N \in \N_0$, a positive constant $C_{r,N}$ so that for each
$\prescript{}{l}{u}\in \prescript{}{l}{PWS}_{\tau,H}(\aL^*_\C)$ with $\|P\prescript{}{l}{u}\|_{r,N} < \infty$, one can find $\prescript{}{l}{v}\in \prescript{}{l}{PWS}_{\tau}(\aL^*_\C)$ satisfying
\begin{itemize}
\item[(i)] $P\prescript{}{l}{u}=P\prescript{}{l}{v}$ and
\item[(ii)] 
$\|\prescript{}{l}{v}\|_{r,N+M} \leq C_{r,N}\|P\prescript{}{l}{u}\|_{r,N}.$
\end{itemize}
The constants $C_{r,N}$ and $M$ can be chosen to be independent of the integer $l$.
\end{thm}

In the proof of the theorem, we will use the following standard result.
For convenience of the reader, we include the proof here.

\begin{lem} \label{lem:PDE2}
Let $p(z)=\sum_{n=0}^k a_n z^n$ be a polynomial in one variable such that the leading coefficient $a_k$ is not zero. 
For $r>1$, let $f$ be a holomorphic function on $B_r(0) \subset \aL^*_\C$. 
Then
\begin{equation} \label{eq:estimp}
|f(0)| \leq {|a_k|}^{-1} \sup_{|z|=1} |f(z)p(z)|.
\end{equation}
Here, $B_r(0)$ denotes an open ball of radius $r$ centered at $0$ in $\aL^*_\C$.
\end{lem}
\noindent
Note that \eqref{eq:estimp} implies that for $f\in \Hol(\aL^*_\C)$
\begin{equation} \label{eq:Lem6}
|f(\lambda)| \leq C_p \sup_{|z|=1} |f(\lambda+z)p(\lambda+z)|, \;\; \lambda \in \aL^*_\C,
\end{equation}
where the constant $C_p$ depends on the polynomial $p$, but not on $\lambda \in \aL^*_\C$.

\begin{proof}[Proof of Lem.~\ref{lem:PDE2}]
Set $q(z):=\sum_{n=0}^k \overline{a}_{k-n} z^n$. The product $q \cdot f$ is holomorphic on $B_r(0)$. By the maximum principle for $B_1(0)$, we have
$$|q(0) f(0)| \leq \sup_{|z|=1} |q(z) f(z)|.$$
Note that $q(0)=\overline{a}_k$.
 If $|z|=1$, then we obtain
$$
\overline{p(z)}=\sum_{n=0}^k \overline{a}_n \overline{z}^n = \sum_{n=0}^k \overline{a}_n z^{-n} =z^{-k} \sum_{n=0}^\infty \overline{a}_n z^{k-n} =z^{-k}q(z).
$$
Therefore, for $|z|=1$, we have $|p(z)| =|q(z)|,$ and thus the lemma follows.
\end{proof}

The strategy of the proof of Thm.~\ref{thm:Hyp23SL2R} is borrowed from Hörmander's proof of an analogous (local) result for the Euclidean case \cite[Prop.~7.6.5]{Hormander}.
Our arguments are simplier and already give global estimates since we do not have to consider functions of several complex variables.
On the other hand, our situation is bit more complicated due to the presence of the generators $q_{n,m}$.

\begin{proof}[Proof of Thm.~\ref{thm:Hyp23SL2R}]
We proceed by induction on the dimension $d_\gamma$ of the vector space $E_\gamma$. 

\textit{Initial case}: Let us show that the theorem is true when $d_\gamma=1$, i.e. $\gamma=m$ for some $m \in \Z$.\\
Set $\prescript{}{l}{w}:= P \prescript{}{l}{u}=\sum_{j=1}^{d_\tau} P_j\prescript{}{l}{u}_j$,
where $P_j \in \prescript{}{j}{PWS}_{m,0}(\aL^*_\C)$ and $\prescript{}{l}{u}_j \in \prescript{}{l}{PWS}_{n_j,H}(\aL^*_\C)$.
Now by Thm.~\ref{thm:Meta3SL2R}, we have
\begin{eqnarray}
\prescript{}{l}{u}_j  &=& \prescript{}{l}{h}_j \cdot q_{l,n_j} \label{eq:lu} \\
P_j &=& a_j \cdot q_{n_j,m}  \label{eq:P} \\
\prescript{}{l}{w} &=& \prescript{}{l}{b} \cdot q_{l,m} \label{eq:lw}
\end{eqnarray}
for some $\prescript{}{l}{h}_j, \prescript{}{l}{b} \in \Hol(\lambda^2)$ and $a_j \in \Pol(\lambda^2)$.
In order to obtain the desired $\prescript{}{l}{v}$, we want to replace each $\prescript{}{l}{h}_j$ by some $\prescript{}{l}{\tilde{h}}_j$, which can be estimated by $P$ and $\prescript{}{l}{w}$.

We now observe by putting the equations \eqref{eq:lu}, \eqref{eq:P} and \eqref{eq:lw} in $\sum_{j=1}^{d_\tau} P_j\prescript{}{l}{u_j}=\prescript{}{l}{w}$, for $\lambda \in \aL^*_\C$:
\begin{eqnarray*}
\sum_{j=1}^{d_\tau} P_j(\lambda) \prescript{}{l}{u}_j(\lambda)=\prescript{}{l}{w}(\lambda)
&\iff&
\sum_{j=1}^{d_\tau} a_j(\lambda) q_{n_j,m}(\lambda) q_{l,n_j}(\lambda) \prescript{}{l}{h}_j(\lambda) = \prescript{}{l}{b}(\lambda) q_{l,m}(\lambda) \\
&\stackrel{\text{Lem.~\ref{lem:rlnm}}}{\iff}& \sum_{j=1}^{d_\tau} a_j(\lambda) r^l_{n_j,m}(\lambda) \prescript{}{l}{h}_j(\lambda) = \prescript{}{l}{b}(\lambda)\\
&\iff&  \sum_{j=1}^{d_\tau} \prescript{}{l}{\tilde{a}}_j(\lambda) \prescript{}{l}{h}_j(\lambda) = \prescript{}{l}{b}(\lambda),
\end{eqnarray*}
where in the last line we set $\prescript{}{l}{\tilde{a}}_j(\lambda):=a_j(\lambda) r^l_{n_j,m}(\lambda), \forall j$.
This implies that $\prescript{}{l}{b}$ is divisible by the greatest common divisor
$$\text{gcd}_{\C[\lambda^2]}(\prescript{}{l}{\tilde{a}}_1(\lambda), \dots, \prescript{}{l}{\tilde{a}}_{d_\tau}(\lambda))=:\prescript{}{l}{\tilde{p}}(\lambda) \in \Pol(\lambda^2),$$
i.e.
\begin{equation} \label{eq:lk}
\prescript{}{l}{b}(\lambda)=\prescript{}{l}{\beta}(\lambda)\prescript{}{l}{\tilde{p}}(\lambda)
\end{equation}
for some $\prescript{}{l}{\beta} \in \Hol(\lambda^2)$. 
Since $\Pol(\lambda^2)$ is a principal ideal domain, by Bézout's theorem, we can find polynomials $\prescript{}{l}{R}_j \in \Pol(\lambda^2)$ such that
$$ \sum_{j=1}^{d_\tau} \prescript{}{l}{\tilde{a}}_j(\lambda) \prescript{}{l}{R}_j(\lambda)=\prescript{}{l}{\tilde{p}}(\lambda).$$
Now we set
$$\prescript{}{l}{\tilde{h}}_j(\lambda):= \prescript{}{l}{\beta}(\lambda)  \prescript{}{l}{R}_j(\lambda), \;\; \text{ for } \lambda \in \aL^*_\C, j=1, \dots, d_\tau$$
and define  $\prescript{}{l}{v}_j\in \prescript{}{l}{PWS}_{n_j,H}(\aL^*_\C)$ by
\begin{equation} \label{eq:newlv}
\prescript{}{l}{v}_j(\lambda) := \prescript{}{l}{\tilde{h}}_j(\lambda) q_{l,n_j}(\lambda)=\prescript{}{l}{\beta}(\lambda)  \prescript{}{l}{R}_j(\lambda)q_{l,n_j}(\lambda).
\end{equation}
Concerning the estimate, from the equations \eqref{eq:lw}, \eqref{eq:lk}, \eqref{eq:newlv} and the relation \eqref{eq:rlnm}, we obtain
$$ r^l_{n_j,m}(\lambda)\prescript{}{l}{R}_j(\lambda)\prescript{}{l}{w}(\lambda) =\prescript{}{l}{v}_j(\lambda)\prescript{}{l}{\tilde{p}}(\lambda)q_{n_j,m}(\lambda).$$
Due \eqref{eq:Lem6}, this leads to
$$
|\prescript{}{l}{v}_j(\lambda)| \leq C_{l,j}^1 \sup_{|z| \leq 1} \Big\{ |r^l_{n_j,m}(\lambda +z) \prescript{}{l}{R}_j(\lambda+z) \prescript{}{l}{w}(\lambda+z)| \Big\},
$$
where $C_{l,j}^1$ is a constant depending on the two polynomials $\prescript{}{l}{\tilde{p}}$ and $q_{n_j,m}$.
Moreover, since $r^l_{n_j,m}$ and $\prescript{}{l}{R}_j$ are polynomials, we can choose $M_l \in \N_0$ so that
$$ |r^l_{n_j,m}(\lambda) \prescript{}{l}{R}_j(\lambda)| \leq C^2_l(1+|\lambda|^2)^{M_l}, \;\; \lambda \in \aL^*_\C.$$
The constants $C_{l,j}^1, C^2_l, M_l$ depend on the integer $l$.
However, by Lem.~\ref{lem:rlnm}, there are only finitely many $r^l_{n_j,m}$ and therefore also $\prescript{}{l}{R}_j$ and $\prescript{}{l}{\tilde{p}}$.
Hence, we can choose these constants such that
$C:= \sup_{l \in \Z, j} C^1_{l,j} \cdot C^2_l$ and $M:=\sup_{l \in \Z} M_l< \infty$.
Coming back to our inequality, we get
\begin{eqnarray*}
\|\prescript{}{l}{v}(\lambda)\|_{\text{op}} 
= \max\{|\prescript{}{l}{v}_j(\lambda)| \;|\; j=1, \dots, d_\tau\} 
&\leq& C \sup_{|z| \leq 1} \Big\{(1+|\lambda+z|^2)^M |\prescript{}{l}{w}(\lambda+z)|\Big\}\\
&\leq& C' (1+|\lambda|^2)^{M} \sup_{|z| \leq 1} |\prescript{}{l}{w}(\lambda+z)|,
\end{eqnarray*}
where $C'>0$ is a constant independent of $l$.
Now by multiplying both sides by the weight factor $(1+|\lambda|^2)^{-(N+M)}e^{-r|\Repart(\lambda)|}$ and taking the supremum over $\lambda$, we obtain
\begin{eqnarray*}
\|\prescript{}{l}{v}\|_{r,N+M} 
\leq C' \sup_{\lambda \in \C} (1+|\lambda|^2)^{-N} e^{-|\Repart(\lambda)|} |\prescript{}{l}{w}(\lambda)| 
&\leq& C'' \sup_{\lambda \in \C, |z| \leq 1} (1+|\lambda+z|^2)^{-N} e^{-|\Repart(\lambda+z)|} |\prescript{}{l}{w}(\lambda+z)| \\
&=& C'' \|\prescript{}{l}{w}\|_{r,N}.
\end{eqnarray*}
This is the desired estimate for $d_\gamma=1$.
Note that $C''$ is independent of $l$.

\textit{Inductive step:} It remains to show that the theorem is true for $d_{\gamma}>1$.\\
By induction hypothesis, assume that the statement is already proved for systems involving less than $d_\gamma$ equations.
Write $P=(P_1, P')^T$, where $P_1$ stands for the first and $P'$ for the $d_{p-1}$ remaining rows of the matrix $P$.\\
In particular, we then can consider the equation $P_1 \prescript{}{l}{v}_1=P_1 \prescript{}{l}{u}$ and conclude that it has a solution  $\prescript{}{l}{v}_1 \in \prescript{}{l}{PWS}_{n}(\aL^*_\C)$ such that
\begin{equation} \label{eq:Estimatev1}
\|\prescript{}{l}{v}_1\|_{r,N+M} \leq C \|P_1\prescript{}{l}{u}\|_{r,N} \leq C \|P \prescript{}{l}{u}\|_{r,N}
\end{equation}
with constants $C, M$ independent of $l$.
Now, we make the ansatz $\prescript{}{l}{v}=\prescript{}{l}{v}_1+\prescript{}{l}{h}$, where we have to find $\prescript{}{l}{h}$ so that $P_1 \prescript{}{l}{h}=0, P'\prescript{}{l}{h}=P'(\prescript{}{l}{u}-\prescript{}{l}{v}_1)$ and $\prescript{}{l}{h}$ can be estimated.\\
By applying Thm.~\ref{thm:Hyp1SL2R} to $P_1$ instead of $P$ and using $P_1 (\prescript{}{l}{u}-\prescript{}{l}{v}_1)=0$, we can write
\begin{equation} \label{eq:uvQ}
\prescript{}{l}{u}-\prescript{}{l}{v}_1=Q_1 \prescript{}{l}{f},
\end{equation}
where $\prescript{}{l}{f} \in \prescript{}{l}{PWS}_{\delta_1,H}(\aL^*_\C)$.
We want to find $\prescript{}{l}{h}$ of the form $\prescript{}{l}{h}=Q_1 \prescript{}{l}{g}$.\\
Moreover, the equations $P' \prescript{}{l}{h}=P' (\prescript{}{l}{u}-\prescript{}{l}{v}_1)$ then become $$P' Q_1 \prescript{}{l}{g}=P' Q_1\prescript{}{l}{f}.$$
Let $d_1 \geq \frac{\deg P'}{2}$ and set $N':=N+M+d_1$. 
Since, we only have $d_{\gamma}-1$ equations, by induction hypothesis there exist some constants $M'$ and $C'_{r,N}$ (independent of $l$), and
$\prescript{}{l}{g}  \in \prescript{}{l}{PWS}_{\delta}(\aL^*_\C)$ with 
\begin{eqnarray} \label{eq:Estimateg}
\|\prescript{}{l}{g}\|_{r,N'+M'} 
\leq C'_{r,N'} \|P'Q\prescript{}{l}{f}\|_{r,N'} 
&\stackrel{\eqref{eq:uvQ}}{=}& C'_{r,N'} \|P'(\prescript{}{l}{u}-\prescript{}{l}{v_1})\|_{r,N'} \nonumber \\
&\leq& {C'}_{r,N'} \Big(\|P'\prescript{}{l}{u}\|_{r,N'} + \| P'\prescript{}{l}{v_1}\|_{r,N'}\Big) \nonumber \\
&\leq& {C'}_{r,N'}\Big(\|P\prescript{}{l}{u}\|_{r,N'} +C' \|v_1\|_{r,N+M}\Big) \nonumber \\
&\stackrel{\eqref{eq:Estimatev1}}{\leq}& C''_{r,N'} \|P \prescript{}{l}{u}\|_{r,N'}. 
\end{eqnarray}
Now let $d_2 \geq \frac{\deg Q_1}{2}$ and set $M'':=M+M'+d_1+d_2$.
We obtain for $\prescript{}{l}{v}:=\prescript{}{l}{v}_1+Q\prescript{}{l}{g}$
\begin{eqnarray*}
\|\prescript{}{l}{v}\|_{r,N+M''}
\leq \|\prescript{}{l}{v}_1\|_{r,N+M''} +\|Q\prescript{}{l}{g}\|_{r,N+M''} 
&\leq& \|\prescript{}{l}{v}_1\|_{r,N+M} + C'' \|\prescript{}{l}{g}\|_{r,N'+M'} \\
&\stackrel{\eqref{eq:Estimatev1} \;\&\; \eqref{eq:Estimateg}}{\leq}& C''_{r,N} \|P\prescript{}{l}{u}\|_{r,N}.
\end{eqnarray*}
This is the desired estimate.
\end{proof}

By the results of Sect.~\ref{sect:H2}, we conclude that Conj.~\ref{conj:1} and its variants (see Sect.~\ref{sect:variants}) are true for $G/K=\HH^2$.\\
Let us state this as a theorem.

\begin{thm} \label{thm:new1}
Let $G=\SL(2,\R)$ and $D \in \DD_G(\E_\gamma, \E_\tau)$ be an $G$-invariant differential operator between sections of homogeneous vector bundles over the hyperbolic plane $\HH^2$.
Then, $D$ satisfies all 4 variants of solvability stated in Def.~\ref{defn:variants}. \qed
\end{thm}

\section{Going beyond} \label{sect:Beyond}
\subsection{Solvability on finite products of $\HH^2$}
Let $G:=G'\times G' = \SL(2,\R)\times \SL(2,\R)$ and $K:=K'\times K'$ with $K'=\SO(2)$.
$K$ is maximal compact in $G$.
Analogous to Sect.~\ref{sect:H2}, the symmetric space $X=G/K$ can be identified with $\HH^2 \times \HH^2$, and $\aL^*_\C$ with $\C \times \C$.
We also have $\widehat{K} \cong \Z \times \Z$, where for $n=(n_1,n_2) \in \Z \times \Z$ we have $E_n = \C$ with $K$-action $(k_{\theta_1}, k_{\theta_2}) \mapsto e^{in_1\theta_1} e^{in_2\theta_2} \in \mathbf{U}(1)$ (compare Sect.~\ref{sect:H2}).
Consider now an additional irreducible $K$-representation $(m,E_m)$.
Note that $\Hom_M(E_n,E_m)=0$ unless $n_1 \equiv m_1$ (mod 2), $n_2 \equiv m_2$ (mod 2).
In the latter case, we identify $\Hom_M(E_n,E_m)$ with $\C$ and define $q_{n,m} \in \Pol(\aL^*_\C, \Hom(E_n,E_m))$ by
$$q_{n,m}(\lambda)= q_{n_1,m_1} (\lambda_1) \cdot q_{n_2,m_2}(\lambda_2)$$
for $\lambda=(\lambda_1,\lambda_2) \in \aL^*_\C$ (see \eqref{eq:qnm}).
In \cite[Sect.~4]{PalmirottaIntw}, we described and presented a complete Paley-Wiener-Schwartz theorem for $G$.

\begin{thm}[Intertwining condition in (\textcolor{blue}{Level 3}), \cite{PalmirottaIntw}, Thm.~7] \label{thm:Meta3SL4R}
Let $n,m \in \Z^2$ be two tuples of integers.
Then, $\varphi \in \emph \Hol(\aL^*_\C,\emph \Hom_M(E_n,E_m))$ satisfies the intertwining condition (3.ii) of Def.~\ref{def:PWSspace} if and only if there exists an holomorphic function $h\in \emph \Hol(\lambda_1^2,\lambda_2^2)$, i.e.
$h(\lambda_1,\lambda_2)=h(-\lambda_1,\lambda_2)=h(\lambda_1,-\lambda_2)$
such that 
\begin{equation} \label{eq:compcondLevel3}
\varphi(\lambda_1,\lambda_2):=h(\lambda_1,\lambda_2) \cdot q_{l,n}(\lambda_1,\lambda_2). 
\end{equation} \qed
\end{thm}

Let $\tau, \gamma$ be two, not necessarily irreducible, representations of $K$ and $D \in \DD_G(\E_{\tilde{\gamma}}, \E_{\tilde{\tau}})$. 
We first prove Hyp.~\ref{lem:PQ0L3}.

\begin{thm} \label{thm:Hyp1SL4R}
Let $P, Q$ be as in Hyp.~\ref{lem:PQ0L3}.
The sequence
\begin{equation} \label{eq:shortseqSL4R}
 \prescript{}{l}{PWS}_{\delta,H}(\aL^*_\C) \stackrel{Q}{\longrightarrow} \prescript{}{l}{PWS}_{\tau,H}(\aL^*_\C) \stackrel{P}{\longrightarrow} \prescript{}{l}{PWS}_{\gamma,H}(\aL^*_\C)
\end{equation}
is exact in the middle, that means that  $\emph \Impart(Q)=\emph \Ker(P)$.
\end{thm}

\begin{proof}
Using Thm.~\ref{thm:Meta3SL4R}, we obtain as in the proof of Thm.~\ref{thm:Hyp1SL2R}
$$\prescript{}{l}{PWS}_{*,H}(\aL^*_\C) \cong \Hol(\lambda_1^2,\lambda_2^2) \otimes_{\Pol(\lambda_1^2,\lambda_2^2)}
\prescript{}{l}{PWS}_{*,0}(\aL^*_\C), \;\; *=\delta,\tau, \gamma.$$
We can proceed as in the proof of Thm.~\ref{thm:Hyp1SL2R},
except that here $\Pol(\lambda_1^2,\lambda_2^2)$ is not a principal ideal domain.
However, by using \cite[Lem.~7.6.4]{Hormander}, we see that
$\Hol(\lambda_1^2,\lambda_2^2)$ is still a flat module over $\Pol(\lambda_1^2,\lambda_2^2)$.
\end{proof}

Next, we deal with Hyp.~\ref{thm:estimateL3}.

\begin{thm}[Estimate result in (\textcolor{blue}{Level 3})] \label{thm:Hyp2SL4R}
Fix $l=(l_1,l_2) \in \Z^2$ and let $P\in \prescript{}{\tau}{PWS}_{\gamma,0}(\aL^*_\C)$ as in Thm.~\ref{thm:Hyp1SL4R}.
Then, there exist constants $M\in \N_0$ and, for all $r\geq 0$, $N\in \N_0$, $C_{r,N} >0$ such that for each $\prescript{}{l}{u}\in \prescript{}{l}{PWS}_{\tau,H}(\aL^*_\C)$ with
$\|P\prescript{}{l}{u}\|_{r,N}< \infty$
one can find $\prescript{}{l}{v}\in \prescript{}{l}{PWS}_{\tau}(\aL^*_\C)$ satisfying
\begin{itemize}
\item[(i)] $P\prescript{}{l}{u}=P\prescript{}{l}{v}$ and
\item[(ii)]
$\|\prescript{}{l}{v}\|_{r,N+M} \leq C_{r,N}\|P\prescript{}{l}{u}\|_{r,N}.$
\end{itemize}
\end{thm}

Before we start to prove Thm.~\ref{thm:Hyp2SL4R}, we need to adapt Lem.~\ref{lem:PDE2} for our situation.\\
Consider a polynomial $p$ in two variables $z=(z_1,z_2) \in \C \times \C$.
It can be decomposed into homogeneous components
$p = \sum_{l=0}^k p_k,$ where $p_k \neq 0.$
For all $s\in \C\backslash\{0\}$, we have 
\begin{equation} \label{eq:rulepsz}
p_l(s\cdot z)=s^l p_l(z), \;\; \forall l=0, \dots, k.
\end{equation}
If $p_k \neq 0$, then there exists $v \in \C^2\backslash\{0\}$ such that $p_k(v) \neq 0$ and $|v|=1$.
Fix now $\lambda \in \C \times \C$ and let $p_\lambda :  \C \rightarrow  \C$ be defined by
$$p_\lambda(z):=p(\lambda+z\cdot v).$$
\begin{lem} \label{lem:PDE2V2}
The polynomial $p_\lambda$ is of degree $k$ with largest coefficient $a_k=a_{k,\lambda} \neq 0$ and independent of $\lambda \in \C \times  \C$.
In fact, $a_k=p_k(v)$.
\end{lem}

\begin{proof}
We have
\begin{eqnarray*}
a_k=\lim_{z \rightarrow \infty} z^{-k} p_\lambda(z)
=\lim_{z \rightarrow \infty} z^{-k} \sum_{l=0}^k p_l(\lambda+z\cdot v)
&=&\lim_{z \rightarrow \infty} z^{-k} \sum_{l=0}^k p_l\Big(z\Big(\frac{\lambda}{z}+v\Big)\Big) \\
& \stackrel{\eqref{eq:rulepsz}}{=} &\lim_{z \rightarrow \infty} \sum_{l=0}^k  z^{l-k} p_l\Big(\frac{\lambda}{z}+v\Big) \\
&=&p_k(v). \qedhere
\end{eqnarray*}
\end{proof}
\noindent
Now, let $f$ be an entire function on $\C^2$.
Set $f_\lambda(z):=f(\lambda+z\cdot v)$.
By Lem.~\ref{lem:PDE2}, we get 
$|f_\lambda(0)| \leq \frac{1}{|a_k|} \sup_{|z|=1} |p_\lambda(z) f_\lambda(z)|$, i.e.
\begin{equation} \label{eq:LemPDE2new}
|f(\lambda)| \leq \underbrace{\frac{1}{|a_k|}}_{=:C_p} \sup_{|z|= 1} |p(\lambda+z\cdot v) f(\lambda+z\cdot v)| \leq C_p \sup_{\mu \in \C^2, |\mu|\leq 1} |p(\lambda+\mu) f(\lambda+\mu)|,
\end{equation}
where the constant $C_p$ depends on the highest homogeneous part of $p$.

\begin{proof}[Proof of Thm.~\ref{thm:Hyp2SL4R}]
We decompose $E_\tau = \bigoplus_{j=1}^{d_\tau} E_{n_j}$ and $E_\gamma = \bigoplus_{i=1}^{d_\gamma} E_{m_i}$ into irreducible ones.
Note that $n_j, m_i$ are now tuples of integers.
Then $\prescript{}{l}{u}$ has components $\prescript{}{l}{u}_j \in \prescript{}{l}{PWS}_{n_j}(\aL^*_\C)$, $\prescript{}{l}{w} := P \prescript{}{l}{u}$ has components $\prescript{}{l}{w}_i \in \prescript{}{l}{PWS}_{m_i}(\aL^*_\C)$, and $P$ is given by a matrix $(P_{ij})$ with $P_{ij} \in \prescript{}{n_j}{PWS}_{m_i,0}(\aL^*_\C).$
Then
\begin{eqnarray*}
\prescript{}{l}{u}_j &=&\prescript{}{l}{h_j} \cdot q_{l,n_j}  \label{eq:lv2} \\
P_{ij} &=& a_{ij} \cdot q_{n_j,m_i} \label{eq:P2} \\
\prescript{}{l}{w}_i &=&\prescript{}{l}{b} \cdot q_{l,m_i}  \label{eq:lw2}
\end{eqnarray*}
for some
$\prescript{}{l}{h}_j, \prescript{}{l}{b}_i \in \Hol(\lambda_1^2,\lambda_2^2)$ and $a_{ij} \in \Pol(\lambda_1^2,\lambda_2^2)$.
We want to find $\prescript{}{l}{v}_j\in \prescript{}{l}{PWS}_{n_j}(\aL^*_\C)$ so that $P\prescript{}{l}{v}=\prescript{}{l}{w},$ and to estimate $\prescript{}{l}{v}$ in terms of $\prescript{}{l}{w}$.
As in \eqref{eq:rlnm}, we define polynomials $r^l_{n,m}$ by
$$q_{n,m} \cdot q_{l,n}= r^l_{n,m} \cdot q_{l,m}.$$
We set $\prescript{}{l}{\tilde{a}_{ij}}:=a_{ij}r^l_{n_j,m_i}$ and obtain $Pu=w \iff \sum_{j=1}^{d_\tau} \prescript{}{l}{\tilde{a}_{ij}} \prescript{}{l}{h_j}= \prescript{}{l}{b_i}, i=1, \dots d_\gamma.$\\
Let $\widetilde{A}:=(\prescript{}{l}{\tilde{a}_{ij}})_{\substack{i=1, \dots, d_\gamma \\ j=1, \dots, d_\tau}}.$
Hörmander's estimate \cite[Thm.~7.6.11]{Hormander} tells us that there exist constants $\tilde{M}, C_{r,N}$ so that for every $\prescript{}{l}{h} \in \Hol(\lambda_1,\lambda_2)^{d_\tau}$ with
$\|\widetilde{A}\prescript{}{l}{h}\|_{r,N}< \infty,$
for some $N\in \N_0$, $r\geq 0$, one can find $\prescript{}{l}{\tilde{h}} \in \Hol(\lambda_1,\lambda_2)^{d_\tau}$ such that
$$\sum_{j=1}^{d_\tau} \prescript{}{l}{\tilde{a}}_{ij}\prescript{}{l}{\tilde{h}}_j= \sum_{j=1}^{d_\tau}\prescript{}{l}{\tilde{a}}_{ij}\prescript{}{l}{h}_j=\prescript{}{l}{b}_i, \;\;\; i=1, \dots, d_\gamma$$ 
and
$ 
\|\prescript{}{l}{\tilde{h}}\|_{r,N+\tilde{M}} \leq C_{r,N}\|\prescript{}{l}{b}\|_{r,N}.
$
We set
$$\prescript{}{l}{\tilde{h}'}_j(\lambda_1,\lambda_2)
:=\frac{1}{4}( \prescript{}{l}{\tilde{h}}_j(\lambda_1,\lambda_2)
+\prescript{}{l}{\tilde{h}}_j(-\lambda_1,\lambda_2)
+\prescript{}{l}{\tilde{h}}_j(\lambda_1,-\lambda_2)
+\prescript{}{l}{\tilde{h}}_j(-\lambda_1,-\lambda_2)).$$
Then, $\prescript{}{l}{\tilde{h}'} \in \Hol(\lambda_1^2, \lambda_2^2)^{d_\tau}$.
Note that $\prescript{}{l}{\tilde{h}'}_j$ is still a solution of $ \sum_{j=1}^{d_\tau} \prescript{}{l}{\tilde{a}_{ij}} \prescript{}{l}{\tilde{h}'}_j=\prescript{}{l}{b_i},$ since the functions $\tilde{a}_{ij}$ and $\prescript{}{l}{b_i}$ are even in $\lambda_1$ and $\lambda_2$.
We also have
$
\|\prescript{}{l}{\tilde{h}'}\|_{r,N'+\tilde{M}'} \leq C_{r,N'}\|\prescript{}{l}{b}\|_{r,N}.
$
By using Lem.~\ref{lem:PDE2V2}, in particular \eqref{eq:LemPDE2new}, we have
$$\|\prescript{}{l}{b}_i\|_{r,N} \leq C_{q_{l,m_i}} \|\prescript{}{l}{b}_i q_{l,m_i}\|_{r,N} =C_{q_{l,m_i}} \|\prescript{}{l}{w}_i\|_{r,N}$$
 and thus $\|\prescript{}{l}{b}\|_{r,N} \leq C_l |\prescript{}{l}{w}\|_{r,N}$
for some $C_l.$
We obtain
\begin{equation} \label{eq:estimatelh}
\|\prescript{}{l}{\tilde{h}'}\|_{r,N+\tilde{M}} \leq C_{l,r,N} \|\prescript{}{l}{w}\|_{r,N}.
\end{equation}
We set
$
\prescript{}{l}{v}_j(\lambda_1,\lambda_2)=\prescript{}{l}{\tilde{h}'}_j(\lambda_1,\lambda_2) \cdot q_{l,n_j}(\lambda_1,\lambda_2)
$
and $M_l:=\tilde{M}+d_l$, where $d_l \geq \frac{\deg q_{l,n_j}}{2}, j=1, \dots, d_\tau$.
Then, $P\prescript{}{l}{v}=\prescript{}{l}{w}$ and we can estimate
$$\|\prescript{}{l}{v}\|_{r,N+M_l} \leq C'_l \|\prescript{}{l}{\tilde{h}}\|_{r,N+\tilde{M}} \stackrel{\eqref{eq:estimatelh}}{\leq} C'_{l,r,N} \|\prescript{}{l}{w}\|_{r,N}. \qedhere $$
\end{proof}

In the above proof the constant $M=M_l$ depends heavily on the $K$-type $l \in \Z^2$ (as do the constants $C_{r,N}=C'_{l,r,N}$).
Thus, the methods used here are not sufficient to establish Hyp.~\ref{hyp:constants}.
For that we would need a refinement of Hörmander's results in the Euclidean case \cite[Sect.~7]{Hormander} involving divisibility by polynomials of the form $q_{l,n}$.
Hence we cannot prove Conj.~\ref{conj:1} in the present case.
However, its $K$-finite analogue is true by Prop.~\ref{prop:1}.\\
Theorems~\ref{thm:Meta3SL4R}-\ref{thm:Hyp2SL4R} immediately generalize to
$$G=\SL(2,\R)^d=\underbrace{\SL(2,\R) \times \cdots \times \SL(2,\R)}_{d \text{ times}}, \;\;\;\; d\geq 2.$$
Summarizing, we obtain the following theorem.
\begin{thm} \label{thm:new2}
Let $G=\SL(2,\R)^d$ and $D \in \DD_G(\E_\gamma, \E_\tau)$ be an $G$-invariant differential operator between sections of homogeneous vector bundles over $\underbrace{\HH^2 \times \cdots \times \HH^2}_{d \text{ times}}$.
Then, $D$ is solvable on $K$-finite sections. \qed
\end{thm}

\subsection{Solvability on the 3-hyperbolic space $\HH^3$}
Let
$G=\SL(2,\C)$ and $K=\SU(2)$ its maximal compact subgroup.
The quotient $X=G/K$ can be identified with hyperbolic 3-space $\HH^3$.
In \cite[Thm.~9, 10 and 14]{PalmirottaIntw}, we have described explicitly Delorme's intertwining conditions for $G=\SL(2,\C)$ for the three levels. 
The description allows to establish Hyp.~\ref{lem:PQ0L3} and Hyp.~\ref{hyp:constants} in a similar - but in case of Hyp.~\ref{hyp:constants} much more complicated - way as for $G=\SL(2, \R)$.
Here, we will only discuss Hyp.~\ref{lem:PQ0L3}.
The discussion of Conj.~\ref{conj:1} and its variants is left to a forthcoming paper, where in fact all real hyperbolic spaces $\HH^n, n \geq 3$, will be treated.\\

Let $\mu, \tau \in \widehat{K}$ as in \cite{PalmirottaIntw}, we equipped $ \prescript{}{\mu}{PWS}_{\tau,H}(\aL^*_\C)$ with the structure of a $\Hol(\lambda^2)$-module and showed that it is a free module with finitely many (explicitly given) polynomial generators.
This, in particular, implied the following.

\begin{prop}[\cite{PalmirottaIntw}, Thm.~11 and Cor.~1] \label{cor:Anmflat}
For $\mu, \tau \in \widehat{K}$, we have an isomorphism between the pre-Paley-Wiener-Schwartz space $\prescript{}{\mu}{PWS}_{\tau,H}(\aL^*_\C)$ and 
$\emph \Hol(\lambda^2) \otimes_{\emph \Pol(\lambda^2)} \prescript{}{\mu}{PWS}_{\tau,0}(\aL^*_\C).$ \QEDA
\end{prop}
\noindent
Using Prop.~\ref{cor:Anmflat}, we can now prove Hyp.~\ref{lem:PQ0L3} with the same arguments as in case of $G=\SL(2,\R)$ (see the proof of Thm.~\ref{thm:Hyp1SL2R}).

\begin{thm} \label{thm:Hyp1SL2C}
We consider three, not necessarily irreducible, $K$-representations $(\delta, E_\delta)$, $(\tau,E_\tau)$ and $(\gamma, E_\gamma)$.
Fix $\mu \in \widehat{K}$. 
Then, the sequence 
\begin{equation*}
 \prescript{}{\mu}{PWS}_{\delta,H}(\aL^*_\C) \stackrel{Q}{\longrightarrow} \prescript{}{\mu}{PWS}_{\tau,H}(\aL^*_\C) \stackrel{P}{\longrightarrow} \prescript{}{\mu}{PWS}_{\gamma,H}(\aL^*_\C)
\end{equation*}
is exact in the middle, that means that  $\emph \Impart(Q)=\emph \Ker(P)$. \QEDA
\end{thm}

\subsection*{Acknowledgement}
The second author is supported by the Luxembourg National Research Fund (FNR) under
the project \\ PRIDE15/10949314/GSM.


\end{document}